\documentclass[11pt,dvipsnames]{article}

\usepackage{amsmath,amssymb}
\usepackage{amsfonts,mathrsfs,mathtools,dsfont,mathrsfs,bm,bbm} 
\usepackage{amsthm}

\DeclareMathAlphabet{\mathpzc}{OT1}{pzc}{m}{it}
\DeclareSymbolFontAlphabet{\amsmathbb}{AMSb}%
\usepackage{csvsimple,booktabs}
\usepackage{float,enumitem,subfig}
\usepackage{graphicx,epsfig,xcolor}
\usepackage{balance}

\usepackage[ruled,vlined,linesnumbered]{algorithm2e}
\SetKwInput{kwInit}{Initialization}

\usepackage[hyphens]{url}
\usepackage[breaklinks, colorlinks, linkcolor=MidnightBlue, anchorcolor=MidnightBlue, citecolor=MidnightBlue, urlcolor=MidnightBlue]{hyperref}

\usepackage{palatino}

\usepackage{fullpage}

\newcommand{\beq}{\begin{equation}}
\newcommand{\eeq}{\end{equation}}
\newcommand{\beqa}{\begin{eqnarray}}
\newcommand{\eeqa}{\end{eqnarray}}
\newcommand{\beqan}{\begin{eqnarray*}}
\newcommand{\eeqan}{\end{eqnarray*}}

\newcommand{\vnorm}[1]{\left\|#1\right\|}

\newcommand{\ess}{\text{ess}}

\newcommand{\E}{\mathds{E} }

\newcommand{\prob}{\mathbb{P}}

\newcommand{\diag}{\mathop{\mathrm{diag}}}

\newcommand\T{{\mathpalette\raiseT\intercal}}
\newcommand\raiseT[2]{\raisebox{0.25ex}{$#1#2$}
}

\newcommand{\Aset}{\mathbb{A}}

\newcommand{\Rset}{\mathbb{R}}

\newcommand{\Uset}{\mathbb{U}}

\newcommand{\Xset}{\mathbb{X}}

\newcommand{\Zset}{\mathbb{Z}}

\newcommand{\Fcal}{{\cal F}}

\newcommand{\Lcal}{{\cal L}}

\newcommand{\Ncal}{{\cal N}}
\newcommand{\Ocal}{{\cal O}}
\newcommand{\Pcal}{{\cal P}}

\newcommand{\Wcal}{{\cal W}}

\newcommand{\argmin}{\mathop{\rm argmin}}
\newcommand{\argmax}{\mathop{\rm argmax}}

\newcommand{\bone}{\mathbf{1}}

\renewcommand{\v}[1]{{\bm{#1}}}
\newcommand{\ve}{\varepsilon}

\newcounter{l1}
\newcounter{l2}
\newcounter{l3}
\newcommand{\bdotlist}{\begin{list}{$\bullet$}{}}
\newcommand{\bboxlist}{\begin{list}{$\Box$}{}}
\newcommand{\bbboxlist}{\begin{list}{\raisebox{.005in}{{\tiny
$\blacksquare$ \ \ }}}{}}
\newcommand{\bdashlist}{\begin{list}{$-$}{} }
\newcommand{\blist}{\begin{list}{}{} }
\newcommand{\barablist}{\begin{list}{\arabic{l1}}{\usecounter{l1}}}
\newcommand{\balphlist}{\begin{list}{(\alph{l2})}{\usecounter{l2}}}
\newcommand{\bAlphlist}{\begin{list}{\Alph{l2}.}{\usecounter{l2}}}
\newcommand{\bdiamlist}{\begin{list}{$\diamond$}{}}
\newcommand{\bromalist}{\begin{list}{(\roman{l3})}{\usecounter{l3}}}

\newtheorem{theorem}{Theorem}[section]
\newtheorem{lemma}[theorem]{Lemma}
\newtheorem{proposition}[theorem]{Proposition}
\newtheorem{remark}[theorem]{Remark}
\newtheorem*{theoremnn}{Theorem}
\newtheorem{assumption}{Assumption}[section]

\newcommand{\CVaR}{{\rm CVaR}}

\newcommand{\PE}{{{\Pcal}^{\rm E}}}

\newcommand{\pEstar}{{{p}^{\rm E}_\star}}
\newcommand{\dEstar}{{{d}^{\rm E}_\star}}

\newcommand{\PCVaR}{{{\Pcal}^\CVaR}}

\renewcommand{\bone}{\mathds{1}}
\renewcommand{\Wcal}{\mathpzc{W}}

\DeclareMathOperator{\proj}{proj}

\DeclareMathOperator{\indc}{\mathbb{I}}
\renewcommand{\hat}{\widehat}

\renewcommand{\prob}{{\sf Pr}}

\title{A Stochastic Primal-Dual Method for Optimization with Conditional Value at Risk Constraints}
\author{Avinash N. Madavan \qquad Subhonmesh Bose
\thanks{
	A. N. Madavan and S. Bose are at the University of Illinois at Urbana-Champaign Urbana, Illinois, 61801. Emails:  madavan2@illinois.edu, boses@illinois.edu.
}}

\graphicspath{{images/}}

\newcommand{\rev}[1]{{#1}}

\makeatletter
\if@twocolumn
\newcommand{\whencolumns}[2]{#2}
\else
\newcommand{\whencolumns}[2]{#1}
\fi
\makeatother

\makeatletter
\newcommand{\removelatexerror}{\let\@latex@error\@gobble}
\makeatother

\begin{document}

\maketitle


\begin{abstract}
We study a first-order primal-dual subgradient method to optimize risk-constrained risk-penalized optimization problems, where risk is modeled via the popular conditional value at risk (CVaR) measure. The algorithm processes independent and identically distributed samples from the underlying uncertainty in an online fashion, and produces an $\eta/\sqrt{K}$-approximately feasible and $\eta/\sqrt{K}$-approximately optimal point within $K$ iterations with constant step-size, where $\eta$ increases with tunable risk-parameters of CVaR. We find optimized step sizes using our bounds and precisely characterize the computational cost of risk aversion as revealed by the growth in $\eta$. Our proposed algorithm makes a simple modification to a typical primal-dual stochastic subgradient algorithm. With this mild change, our analysis surprisingly obviates the need for a priori bounds or complex adaptive bounding schemes for dual variables assumed in many prior works. We also  draw interesting parallels in sample complexity with that for chance-constrained programs derived in the literature with a very different solution architecture.
\end{abstract}


\section{Introduction}
\label{sec:intro}

We study iterative primal-dual stochastic subgradient algorithms to solve risk-sensitive optimization problems of the form
\whencolumns{
\begin{align}
\begin{aligned}
\PCVaR : \quad
& \underset{\v{x} \in \Xset}{\text{minimize}} && \ F(\v{x}) := \CVaR_\alpha[f_\omega(\v{x})], \\ & \text{subject to } && \  G^i(\v{x}) := \CVaR_{\beta^i}[g^i_\omega(\v{x})] \leq 0, \quad i=1,\ldots, m,
\end{aligned}
\label{eq:prob.P}
\end{align}
}{
\begin{align}
\begin{aligned}
\PCVaR : \quad
& \underset{\v{x} \in \Xset}{\text{minimize}} && F(\v{x}) := \CVaR_\alpha[f_\omega(\v{x})], \\ & \text{subject to } &&  G^i(\v{x}) := \CVaR_{\beta^i}[g^i_\omega(\v{x})] \leq 0, \\ &&& \text{for each } i=1,\ldots, m,
\end{aligned}
\label{eq:prob.P}
\end{align}
}
where $\omega \in \Omega$ is random and $\alpha, \pmb{\beta}:=(\beta^1, \ldots, \beta^m)$ in $[0,1)$ define risk-aversion parameters. The collection of real-valued functions $f_\omega$, $g^1_\omega, \ldots, g^m_\omega$ are assumed convex but not necessarily differentiable, over 
the closed convex set $\Xset \subseteq \Rset^n$,
where $\Rset$ and $\Rset_+$ stand for the set of real and nonnegative numbers, respectively. 
Denote by $\v{G}$ and $\v{g}_\omega$, the collection of $G^i$'s and $g^i_\omega$'s respectively, for $i = 1,\ldots, m$.
$\CVaR$ stands for conditional value at risk. For any $\delta \in [0, 1)$, $\CVaR_\delta[y_\omega]$ of a scalar random variable $y_\omega$ with continuous distribution equals its expectation computed over the $1-\delta$ tail of the distribution of $y_\omega$. For $y_\omega$ with general distributions, $\CVaR$ is defined via the following variational characterization
\begin{align}
\CVaR_{\delta}[y_\omega] = \min_{u \in \Rset} \left\{ u + \frac{1}{1-\delta}\E [y_\omega - u]^+ \right\},
\label{eq:CVaR.def}
\end{align}
following \cite{rockafellar2002optimization}.
\rev{For each $\v{x} \in \Xset$, assume that $\E[ | f_\omega(\v{x}) | ]$ and $\E[ | g^i_\omega(\v{x}) | ]$ are finite, implying that $F$ and $\v{G}$ are well-defined everywhere in $\Xset$.}

$\PCVaR$ offers a modeler the flexibility to indicate her risk preference in $\alpha, \pmb{\beta}$. With $\alpha$ close to zero, she indicates risk-neutrality towards the uncertain cost associated with the decision. With $\alpha$ closer to one, she expresses her risk aversion towards the same and seeks a decision that limits the possibility of large random costs associated with the decision. Similarly, $\beta$'s express the risk tolerance in constraint violation. Choosing $\beta$'s close to zero indicates that constraints should be satisfied on average over $\Omega$ rather than on each sample. Driving $\beta$'s to unity amounts to requiring the constraints to be met almost surely. Said succinctly, $\PCVaR$ permits the modeler to customize risk preference between the risk-neutral choice of expected evaluations of functions to the conservative choice of robust evaluations.


There is a growing interest in solving risk-sensitive optimization problems with data. See \cite{kalogerias2018recursive,bedi2019nonparametric} for recent examples that tackle problems with generalized mean semi-deviation risk that equals $\E[{y_\omega}] + c \E[ \vert {y_\omega} - \E[{y_\omega}] \vert^{p}]^{1/p}$ for $p>1$ for a random variable $y_\omega$. There is a long literature on risk measures, e.g., see \cite{rockafellar2002optimization,ogryczak1999stochastic,ruszczynski2006optimization,charnes1959chance,mafusalov2018buffered,ahmadi2012entropic}. We choose $\CVaR$ for three particular reasons. First, it is a coherent risk measure, meaning that it is normalized, sub-additive, positively homogeneous and translation invariant, i.e., 
\begin{gather*}
\CVaR_{\delta}[0] = 0, \; \CVaR_{\delta}[y^1_\omega + y^2_\omega] \leq  \CVaR_{\delta}[y^1_\omega] + \CVaR_{\delta}[y^2_\omega], \\
\CVaR_{\delta}[t y_\omega] = t \CVaR_{\delta}[y_\omega], \;  \CVaR_{\delta}[y_\omega + t'] = \CVaR_{\delta}[y_\omega] + t'
\end{gather*} 
for random variables $y_\omega, y^1_\omega, y^2_\omega$, $t >0$ and $t' \in \Rset$.
An important consequence of coherence is that $F$ and $\v{G}$ in $\PCVaR$ inherit the convexity of $f_\omega$ and $\v{g}_\omega$. Convexity together with the variational characterization in \eqref{eq:CVaR.def} allow us to design sampling based primal-dual methods for $\PCVaR$ for which we are able to provide finite sample analysis of approximate optimality and feasibility. The popularity of the $\CVaR$ measure is our second reason to study $\PCVaR$. Following Rockafellar and Uryasev's seminal work in \cite{rockafellar2002optimization}, $\CVaR$ has found applications in various engineering domains, e.g., see \cite{miller2017optimal,kisiala2015conditional}, and therefore 
we anticipate wide applications of our result. Our third and final reason to study $\PCVaR$ is its close relation to other optimization paradigms in the literature as we describe next.

$\PCVaR$ without constraints and $\alpha = 0$ reduces to the minimization of $\E[f_\omega(\v{x})]$, the canonical stochastic optimization problem. With $\alpha \uparrow 1$, the problem description of $\PCVaR$ approaches that of a robust optimization problem (see \cite{ben2009robust}) of the form $\min_{\v{x} \in \Xset} \ess\sup_{\omega \in \Omega} f_\omega(\v{x})$, where $\ess\sup$ denotes the essential supremum.  
Driving $\beta$'s to unity, $\PCVaR$ demands the constraints to be enforced almost surely. Such robust constraint enforcement is common in multi-stage stochastic optimization problems with recourse and discrete-time optimal control problems, e.g., in \cite{shapiro2007tutorial,skaf2010design,hadjiyiannis2011efficient}.
$\CVaR$-based constraints are closely related to chance-constraints introduced by Charnes and Cooper in \cite{charnes1959chance}, that enforce 
$ \prob\{g_\omega(\v{x}) \leq 0 \} > 1-\ve$ where $\prob$ refers to the probability measure on $\Omega$. Even if $g_\omega$ is convex, chance-constraints typically describe a nonconvex feasible set. 
It is well-known that $\CVaR$-based constraints provide a convex inner approximation of chance-constraints. 
Restricting the probability of constraint violation does not limit the extent of any possible violation, while $\CVaR$-based enforcement does so in expectation. 
$\CVaR$ is also intimately related to the buffered probability of exceedence (bPOE)  introduced and studied more recently in \cite{mafusalov2018buffered,zhang2019derivatives}. In fact, bPOE is the inverse function of $\CVaR$ and hence, problems with bPOE-constraints can often be reformulated as instances of $\PCVaR$.

It can be challenging to compute $\CVaR$ of $f_\omega(\v{x})$ or $\v{g}_\omega(\v{x})$ for a given decision variable $\v{x}$ with respect to a general distribution on $\Omega$ for two reasons. First, if samples from $\Omega$ are obtained from a simulation tool, an explicit representation of the probability distribution on $\Omega$ may not be available. Second, even if such a distribution is available, computation of $\CVaR$ (or even the expectation) can be difficult. For example, with $f_\omega$ as the positive part of an affine function and $\omega$ being uniformly distributed over a unit hypercube, computation of $\E[f_\omega]$ via a multivariate integral is \#P-hard according to \cite[Corollary 1]{hanasusanto2016comment}. Therefore, we do not assume knowledge of $F$ and $\v{G}$ but rather study a sampling-based algorithm to solve $\PCVaR$.

Solution architectures for $\PCVaR$ via sampling come in two flavors. The first approach is sample average approximation (SAA) that replaces the expectation in \eqref{eq:CVaR.def} by an empirical average over $N$ samples. One can then solve the sampled problem as a deterministic convex program.\footnote{For the unconstrained problem, variance-reduced stochastic gradient descent methods can efficiently minimize the  resulting finite sum as in \cite{schmidt2017minimizing,johnson2013accelerating}.} We take the second and alternate approach of stochastic approximation and process independent and identically distributed (i.i.d.) samples from $\Omega$ in an online fashion. Iterative stochastic approximation algorithms for the unconstrained problem have been studied since the early works by Robbins and Monro in \cite{robbins1971convergence} and by Kiefer and Wolfowitz in \cite{kiefer1952stochastic}. See \cite{kushner2003stochastic} for a more recent survey. Zinkevich in \cite{zinkevich2003online} proposed a projected stochastic subgradient method to tackle constraints in such problems. Without directly knowing $\v{G}$, we cannot easily project the iterates on the feasible set $\{ \v{x} \in \Xset \ | \ \v{G}(\v{x}) \leq 0 \}$. We circumvent the challenge by associating Lagrange multipliers $\v{z} \in \Rset^m_+$ to the constraints and iteratively updating $\v{x}, \v{z}$ by using $f_\omega, \v{g}_\omega$ and their subgradients via a first-order stochastic primal-dual algorithm for $\PCVaR$ along the lines of \cite{nedic2009subgradient,xu2018primal,yu2017online}.

In Section \ref{sec:alg}, we first design and analyze Algorithm \ref{alg:PDSS} for $\PCVaR$ with $\alpha=0, \pmb{\beta} = 0$, i.e., the optimization problem
\whencolumns{
\begin{align}
\begin{aligned}
\PE : \quad
& \underset{\v{x} \in \Xset}{\text{minimize}} && \ F(\v{x}) := \E[f_\omega(\v{x})], \\ & \text{subject to } && \  G^i(\v{x}) := \E[g^i_\omega(\v{x})] \leq 0, \quad i=1,\ldots, m.
\end{aligned}
\label{eq:prob.P.exp}
\end{align}
}{
\begin{align}
\begin{aligned}
\PE : \quad
& \underset{\v{x} \in \Xset}{\text{minimize}} && F(\v{x}) := \E[f_\omega(\v{x})], \\ & \text{subject to } &&  G^i(\v{x}) := \E[g^i_\omega(\v{x})] \leq 0, \\ &&& \text{for each } i=1,\ldots, m.
\end{aligned}
\label{eq:prob.P.exp}
\end{align}
}
First-order stochastic primal-dual algorithms have a long history, dating back almost forty years, including that in \cite{nesterov1998introductory,nedic2009subgradient,ermoliev1976methods,kushner2003stochastic,nemirovski2009robust}. The analyses of these algorithms often require a bound on the possible growth of the dual variables. Borkar and Meyn in \cite{borkar2000ode} stress the importance of compactness assumptions in their analysis of stochastic approximation algorithms. 
A priori bounds used in \cite{nemirovski2009robust} are difficult to know in practice and techniques for iterative construction of such bounds as in \cite{nedic2009subgradient} require extra computational effort.
\rev{A regularization term in the dual update has been proposed in \cite{mahdavi2012trading,koppel2017proximity} to circumvent this limitation. Instead, }
we propose a different modification to the classical primal-dual stochastic subgradient algorithm. We were surprised to find that, with this  simple modification, the analysis allows us to completely bypass the need for explicit or implicit bounds on the dual variable in our analysis. 
While the classical primal-dual approach samples once for a single update of the primal and the dual variables, we sample twice--once to update the primal variable and then again to update the dual variable with the most recent primal iterate--thus, adopting a Gauss-Seidel approach in place of a Jacobi framework.
For Algorithm \ref{alg:PDSS}, we bound the expected optimality gap and constraint violations at a suitably weighted average of the iterates by $\eta/\sqrt{K}$ for a constant $\eta$ with a constant step-size algorithm. Using these bounds, we then carefully optimize the step-size that allows us to reach within a given threshold of suboptimality and constraint violation with the minimum number of iterations.
The additional sample required in our update aids in the analysis; however, it comes at the price of making the sample complexity double of the iteration complexity.
We also provide stability analysis of our algorithm with decaying step-sizes that takes advantage of a dissipation inequality that we derive for our Gauss-Seidel approach.

In Section \ref{sec:cvar}, we solve $\PCVaR$ with general risk aversion parameters $\alpha, \pmb{\beta}$ using Algorithm \ref{alg:PDSS} on an instance of $\PE$ obtained through a standard reformulation via the variational formula for $\CVaR$ in \eqref{eq:CVaR.def} from \cite{rockafellar2002optimization}. 
We then bound the expected suboptimality and constraint violation at a weighted average of the iterates for $\PCVaR$ by $\eta(\alpha, \v{\beta})/\sqrt{K}$. Upon utilizing the optimized step-sizes from the analysis of $\PE$, we are then able to study the precise growth in the required iteration (and sample) complexity of $\PCVaR$ as a function of $\alpha, \v{\beta}$. Not surprisingly, the more risk-averse a problem one aims to solve, the greater this complexity increases. A modeler chooses risk aversion parameters primarily driven by attitudes towards risk in specific applications. Our precise characterization of the growth in sample complexity with risk aversion will permit the modeler to balance between desired risk levels and computational challenges in handling that risk. 
We remark that the algorithmic architecture for the risk neutral problem may not directly apply to the risk-sensitive variant for general risk measures. For example, the algorithm described in \cite{kalogerias2018recursive} for general mean-semideviation-type risk measures is considerably more complex than that required for the risk-neutral problem. We are able to extend our algorithm and its analysis for $\PE$ to $\PCVaR$, thanks to the variational form in \eqref{eq:CVaR.def} that $\CVaR$ admits. See the discussion after the proof of Theorem \ref{thm:cvar} for a precise list of properties a risk measure must exhibit for us to apply the same trick.
Using concentration inequalities, we also report an interesting connection of our results to that in \cite{calafiore2005uncertain,campi2008exact} on scenario approximations to chance-constrained programs. The resemblance in sample complexity is surprising, given that the approach in \cite{calafiore2005uncertain,campi2008exact} solves a deterministic convex program with sampled constraints, while we process samples in an online fashion. 

\rev{We illustrate properties of our algorithm through a stylized example. Our experiments reveal that the optimized iteration count (and sample complexity) for even a simple example is quite high. This limitation is unfortunately common for subgradient algorithms and likely cannot  be overcome in optimizing general nonsmooth functions that we study. While the bounds are order-optimal, our numerical experiments reveal that a solution with desired risk tolerance can be found in less iterations than obtained from the upper bound. This is an artifact of optimizing step-sizes based on upper bounds on suboptimality and constraint violation. 
We end the paper in Section \ref{sec:conc} with discussions on possible extensions of our analysis.}

Very recently, it was brought to our attention that the work in \cite{boob2019stochastic} done concurrently presents a related approach to tackle optimization of composite nonconvex functions under related but different assumptions. In fact, their work appeared at the same time as our early version and claims a similar result that does not require bounds on the dual variables. Our analysis does not require or analyze the case with strongly convex functions within our setup and therefore Nesterov-style acceleration remains untenable. As a result, our algorithm is different. Our focus on $\CVaR$ permits us to further analyze the growth in optimized sample complexity with risk aversion and its connection to chance-constrained optimization, that is quite different.


\section{Algorithm for \texorpdfstring{$\PE$}{PE} and its analysis}
\label{sec:alg}

We present the {primal-dual stochastic subgradient} method to solve $\PE$ in Algorithm \ref{alg:PDSS}. 

\whencolumns{
\begin{algorithm}
	\kwInit{Choose $\v{x}_1 \in \Xset$, $\v{z}_1 = 0$, and a positive sequence $\v{\gamma}$.}
	\For{$k \geq 1$}{
		Sample $\omega_k \in \Omega$. Update  $\v{x}$ as
		\begin{alignat}{1}
			\v{x}_{k+1} 
			&\gets 
			\argmin_{\v{x} \in \Xset}  
				\left\langle \nabla f_{\omega_k}(\v{x}_k) + \sum_{i=1}^m {z}_k^i \nabla {g}_{\omega_k}^i(\v{x}_k), \v{x} - \v{x}_k \right\rangle 
				+ \frac{1}{2\gamma_k} \vnorm{\v{x} - \v{x}_k}^2. 
			\label{eq:xUpdate}
		\end{alignat}

		Sample $\omega_{k+1/2} \in \Omega$. Update $\v{z}$ as
		\begin{alignat}{1}
			\v{z}_{k+1} 
			&\gets 
			\argmax_{\v{z} \in \Rset_+^m} 
				\left\langle \v{g}_{\omega_{k +1/2}}(\v{x}_{k+1}), \v{z} - \v{z}_k \right\rangle 
				- \frac{1}{2\gamma_k} \| \v{z} - \v{z}_k \|^2. 
			\label{eq:zUpdate}
		\end{alignat}
	}
	\caption{Primal-dual stochastic subgradient method for $\PE$.}
	\label{alg:PDSS}
\end{algorithm}
}{
\begin{figure*}[t]
\removelatexerror
\begin{algorithm}[H]
	\kwInit{Choose $\v{x}_1 \in \Xset$, $\v{z}_1 = 0$, and a positive sequence $\v{\gamma}$.}
	\For{$k \geq 1$}{
		Sample $\omega_k \in \Omega$. Update  $\v{x}$ as
		\begin{alignat}{1}
			\v{x}_{k+1} 
			&\gets 
			\argmin_{\v{x} \in \Xset}  
				\left\langle \nabla f_{\omega_k}(\v{x}_k) + \sum_{i=1}^m {z}_k^i \nabla {g}_{\omega_k}^i(\v{x}_k), \v{x} - \v{x}_k \right\rangle 
				+ \frac{1}{2\gamma_k} \vnorm{\v{x} - \v{x}_k}^2. 
			\label{eq:xUpdate}
		\end{alignat}

		Sample $\omega_{k+1/2} \in \Omega$. Update $\v{z}$ as
		\begin{alignat}{1}
			\v{z}_{k+1} 
			&\gets 
			\argmax_{\v{z} \in \Rset_+^m} 
				\left\langle \v{g}_{\omega_{k +1/2}}(\v{x}_{k+1}), \v{z} - \v{z}_k \right\rangle 
				- \frac{1}{2\gamma_k} \| \v{z} - \v{z}_k \|^2. 
			\label{eq:zUpdate}
		\end{alignat}
	}
	\caption{Primal-dual stochastic subgradient method for $\PE$.}
	\label{alg:PDSS}
\end{algorithm}
\end{figure*}
}
The notation $\langle \cdot, \cdot \rangle$ stands for the usual inner product in Euclidean space and $\| \cdot \|$ denotes the induced $\ell_2$-norm. \rev{Here, $\nabla h(\v{x})$ stands for a subgradient of an arbitrary convex function $h$ at $\v{x}$. For our analysis, the subgradient in Algorithm \ref{alg:PDSS} for functions $f_\omega(\v{x})$ and $g_\omega(\v{x})$ can be arbitrary elements of the closed convex subdifferential sets $\partial f_\omega(\v{x})$ and $\partial g_\omega(\v{x})$, respectively. We assume that these subdifferential sets are nonempty everywhere in $\Xset$.} 

\rev{The primal-dual method in Algorithm \ref{alg:PDSS} leverages Lagrangian duality theory. Specifically, define the Lagrangian function for $\PE$ as
\begin{align}
\Lcal(\v{x}, \v{z}) := F(\v{x}) + \v{z}^{\T} \v{G}(\v{x}) = \E[ \Lcal_\omega(\v{x}, \v{z})],
\label{eq:P.L}
\end{align}
for $\v{x} \in \Xset$, $\v{z} \in \Rset^m_+$, where
\whencolumns{$\Lcal_\omega(\v{x}, \v{z}) \coloneqq f_\omega(\v{x}) + \v{z}^{\T} \v{g}_\omega(\v{x})$.}
{$$\Lcal_\omega(\v{x}, \v{z}) \coloneqq f_\omega(\v{x}) + \v{z}^{\T} \v{g}_\omega(\v{x}).$$}
Then, $\PE$ admits the standard reformulation as a min-max problem of the form 
\begin{align}
    \pEstar \coloneqq \min_{\v{x} \in \Xset} \max_{\v{z}\in\Rset^m_+} \Lcal(\v{x}, \v{z}).
    \label{eq:PE.primal}
\end{align}
Denote its optimal set by $\Xset_\star \subseteq \Xset$.
Define the dual problem of $\PE$ as
\begin{align}
    \dEstar \coloneqq \max_{\v{z}\in\Rset^m_+} \min_{\v{x} \in \Xset}  \Lcal(\v{x}, \v{z}).
    \label{eq:PE.dual}
\end{align}
Denote its optimal set by $\Zset_\star \subseteq \Rset^m_+$.
Weak duality then guarantees $\pEstar \geq \dEstar$. When the inequality is met with an equality, the problem is said to satisfy strong duality. A point $(\v{x}_\star, \v{z}_\star) \in \Xset \times \Rset^m_+$ is a saddle point of $\Lcal$ if
\begin{align}
    \Lcal(\v{x}_\star, \v{z}) \leq \Lcal(\v{x}_\star, \v{z}_\star) \leq \Lcal(\v{x}, \v{z}_\star)
    \label{eq:saddle}
\end{align}
for all $(\v{x}, \v{z}) \in \Xset \times \Rset^m_+$. The following well-known saddle point theorem (see \cite[Theorem 2.156]{bonnans2013perturbation}) relates saddle points with primal-dual optimal solutions.}
\begin{theoremnn}[Saddle point theorem]
\rev{A saddle point of $\Lcal$ exists if and only if $\PE$ satisfies strong duality, i.e., $\pEstar = \dEstar$. Moreover, the set of saddle points of $\Lcal$ is given by $\Xset_\star \times \Zset_\star$.}
\end{theoremnn}


Our convergence analysis of Algorithm \ref{alg:PDSS} requires $\PE$ to satisfy the following properties.

\begin{assumption}\ 
\begin{enumerate}[label=(\alph*)]
\item Subgradients of $F$ and $\v{G}$ are bounded, i.e.,
\whencolumns{
$\| \nabla F(\v{x}) \| \leq C_F$, $\|\nabla G^i(\v{x}) \| \leq C_G^i$ for each $i = 1, \dots, m$ and all $\v{x} \in \Xset$.
}{
$$
\| \nabla F(\v{x}) \| \leq C_F, \quad 
\|\nabla G^i(\v{x}) \| \leq C_G^i
$$
for each $i = 1, \dots m$ and all $\v{x} \in \Xset$.
}

\item $\nabla f_\omega$ and $\nabla g_\omega^i$ for $i = 1, \dots, m$ have bounded variance, i.e., 
\whencolumns{
$\E \| \nabla f_\omega(\v{x}) - \E [\nabla f_\omega(\v{x})] \|^2 \leq \sigma_F^2$ and $\E \| \nabla g_\omega^i(\v{x}) - \E [\nabla g_\omega^i(\v{x})] \|^2 \leq [\sigma_G^i]^2$ for all $\v{x} \in \Xset$.
}{
for all $\v{x} \in \Xset$,
\begin{figure*}
\begin{aligned}
	\E \| \nabla f_\omega(\v{x}) - \E [\nabla f_\omega(\v{x})] \|^2 &\leq \sigma_F^2, \\
	\E \| \nabla g_\omega^i(\v{x}) - \E [\nabla g_\omega^i(\v{x})] \|^2 &\leq [\sigma_G^i]^2.
\end{aligned}
\end{figure*}
}

\item $\v{g}_\omega(\v{x})$ has a bounded second moment, i.e., $\E \|  g_\omega^i(\v{x}) \|^2 \leq [D_G^i]^2$ for all $\v{x} \in \Xset$.


\item \rev{$\Lcal$ for $\PE$ admits a saddle point $(\v{x}_\star, \v{z}_\star) \in \Xset \times \Rset^m_+$.}

\end{enumerate}
\label{assumptions}
\end{assumption}

The subgradient of $F$ and the variance of its noisy estimate are assumed bounded. Such an assumption is standard in the convergence analysis of unconstrained stochastic subgradient methods. The assumptions regarding $\v{G}$ are similar, but we additionally require the second moment of the noisy estimate of $\v{G}$ to be bounded over $\Xset$. Boundedness of $\v{G}$ in primal-dual subgradient methods has appeared in prior literature, e.g., in \cite{xu2018primal,yu2017online}.
The second moment remains bounded if ${g}^i_\omega$ is uniformly bounded over $\Xset$ and $\Omega$ for each $i$. It is also satisfied if $\v{G}$ remains bounded over $\Xset$ and its noisy estimate has a bounded variance. \rev{Convergence analysis of unconstrained optimization problems typically assumes the existence of a finite optimal solution. We extend that requirement to the existence of a saddle point in the primal-dual setting, which by the saddle point theorem is equivalent to the existence of finite primal and dual optimal solutions. A variety of conditions imply the existence of such a point; the next result delineates two such sufficient conditions in (a) and (b), where (a) implies (b).}
\begin{lemma}[Sufficient conditions for existence of a saddle point]
\rev{$\Lcal$ for $\PE$ admits a saddle point, if either of the following conditions hold:}
\begin{enumerate}[label=(\alph*)]
    \item \rev{$\Xset_\star$ is nonempty, $\pEstar$ is finite and Slater's constraint qualification holds, i.e., there exists $ \v{x}$ in the relative interior of $\Xset$ for which $\v{G}(\v{x}) < 0$.
    }
    \item \rev{$\PE$ admits a finite ($\v{x}_\star, \v{z}_\star) \in \Xset \times \Rset_+^m$ that satisfies the Karush-Kuhn-Tucker (KKT) conditions given by
\begin{gather}
\begin{gathered}
    0 \in \partial_x \Lcal(\v{x}_\star, \v{z}_\star) + \Ncal_{\Xset}(\v{x}_\star), \;
    {G}^i(\v{x}_\star) \le 0, \; z^i_\star G^i(\v{x}_\star) = 0
\end{gathered}
\label{eq:KKT}
\end{gather}
for $i = 1, \dots, m$,
where $\Ncal_{\Xset}(\v{x}_\star)$ denotes the normal cone of $\Xset$ at $\v{x}_\star$.}
\end{enumerate}
\label{lem:saddle}
\end{lemma}
\begin{proof}
    \rev{Part (a) is a direct consequence of \cite[Theorem 1.265]{bonnans2013perturbation}. To prove part (b), notice that \eqref{eq:KKT}  ensures the existence of subgradients $\nabla F(\v{x}_\star) \in \partial F(\v{x}_\star)$, $\nabla {G}^i(\v{x}_\star) \in \partial {G}^i(\v{x}_\star)$, $i=1,\ldots,m$ and $\v{n} \in \Ncal_\Xset(\v{x}_\star)$ for which
    \begin{align}
    \nabla F(\v{x}_\star) + \sum_{i = 1}^m {z}_\star^i \nabla G^i(\v{x}_\star) + \v{n} = 0.
    \end{align}
    Then, for any $\v{x} \in \Xset$, we have
    \whencolumns{
    \begin{align}
    \underbrace{\langle \nabla F(\v{x}_\star), \v{x} - \v{x}_\star \rangle}_{\leq F(\v{x}) - F(\v{x}_\star)}
     + \sum_{i = 1}^m \underbrace{ {z}_\star^i \langle \nabla G^i(\v{x}_\star) , \v{x} - \v{x}_\star \rangle}_{\leq  {z}_\star^i [ G^i(\v{x}) - G^i(\v{x}_\star) ]}
     + \underbrace{\langle \v{n} , \v{x} - \v{x}_\star \rangle}_{\leq 0} = 0.
    \end{align}
    }{
    \begin{align}
    0 &= \langle \nabla F(\v{x}_\star), \v{x} - \v{x}_\star \rangle
     + \sum_{i = 1}^m {z}_\star^i \langle \nabla G^i(\v{x}_\star) , \v{x} - \v{x}_\star \rangle \\
     & \qquad + \langle \v{n} , \v{x} - \v{x}_\star \rangle \\
     &\le F(\v{x}) - F(\v{x}_\star) + \sum_{i = 1}^m {z}_\star^i [ G^i(\v{x}) - G^i(\v{x}_\star) ].
    \end{align}
    }
    The inequalities in the above relation follow from the convexity of $F$ and ${G}^i$'s, nonnegativity of $\v{z}_\star$, and the definition of the normal cone. From the above inequalities, we conclude $\Lcal(\v{x}, \v{z}_\star) \geq \Lcal(\v{x}_\star, \v{z}_\star)$ for all $\v{x} \in \Xset$. Furthermore, for any $\v{z} \geq 0$, we have
    \begin{align}
    \Lcal(\v{x}_\star, \v{z}_\star) - \Lcal(\v{x}_\star, \v{z}) 
    = \v{z}_\star^{\T} \v{G}(\v{x}_\star)  - \v{z}^{\T} \v{G}(\v{x}_\star) \geq 0,
    \end{align}
    where the last step follows from the nonnegativity of $\v{z}$ and \eqref{eq:KKT}, completing the proof.} 
\end{proof}




\rev{We now present our first main result} that provides a bound on the distance to optimality and constraint violation at a weighted average of the iterates generated by the algorithm on $\PE$ under Assumption \ref{assumptions}. Denote by $\v{C}_G$, $\v{D}_G$, and $\v{\sigma}_G$ the collections of $C_G^i$, $D_G^i$, and $\sigma_G^i$ respectively. We make use of the following notation.
\whencolumns{
\begin{gather}
\begin{gathered}
P_1 := 2 \| \v{x}_1 - \v{x}_\star \|^2 +  4\| \bone + \v{z}_\star \|^2, \\
\quad
P_2 := 8 (4C_F^2 + \sigma_F^2) + 2 \| \v{D}_G \|^2, 
\quad
P_3 := 8 m (4 \| \v{C}_G \|^2 + \| \v{\sigma}_G \|^2).
\end{gathered}
\label{eq:P123.def}
\end{gather}
}{
\begin{figure*}
\begin{aligned}
P_1 &:= 2 \| \v{x}_1 - \v{x}_\star \|^2 +  4\| \bone + \v{z}_\star \|^2, \\
P_2 &:= 8 (4C_F^2 + \sigma_F^2) + 2 \| \v{D}_G \|^2, \\
P_3 &:= 8 m (4 \| \v{C}_G \|^2 + \| \v{\sigma}_G \|^2).
\end{aligned}
\end{figure*}
}

\begin{theorem}[Convergence result for $\PE$] 
Suppose Assumption \ref{assumptions} holds.  For a positive sequence $\{\gamma_k\}_{k=1}^K$, if $P_3 \sum_{k=1}^K \gamma_k^2 < 1$, then the iterates generated by Algorithm \ref{alg:PDSS} satisfy
\whencolumns{
\begin{align}
\E[ F(\bar{\v{x}}_{K+1})] - \pEstar
&\leq   
	\frac{1}{4 \sum_{k=1}^K \gamma_k }  
		\left( \frac{P_1 + P_2 \sum_{k=1}^K \gamma_k^2}{1 - P_3 \sum_{k=1}^K \gamma_k^2} \right), 
	\label{eq:FBound}
	\\
\E[G^i(\bar{\v{x}}_{K+1})] 
& \leq
	\frac{1}{4 \sum_{k=1}^K \gamma_k }  
	\left( 
		\frac{P_1 + P_2 \sum_{k=1}^K \gamma_k^2}{1 - P_3 \sum_{k=1}^K \gamma_k^2}
	\right)
	\label{eq:GBound}
\end{align}
}{
\begin{align}
&\E[ F(\bar{\v{x}}_{K+1})] - \pEstar \notag \\
&\quad \leq   
	\frac{1}{4 \sum_{k=1}^K \gamma_k }  
		\left( \frac{P_1 + P_2 \sum_{k=1}^K \gamma_k^2}{1 - P_3 \sum_{k=1}^K \gamma_k^2} \right), 
	\label{eq:FBound}
	\\
&\E[G^i(\bar{\v{x}}_{K+1})] \notag \\
&\quad \leq
	\frac{1}{4 \sum_{k=1}^K \gamma_k }  
	\left( 
		\frac{P_1 + P_2 \sum_{k=1}^K \gamma_k^2}{1 - P_3 \sum_{k=1}^K \gamma_k^2}
	\right)
	\label{eq:GBound}
\end{align}
}
for each $i = 1, \dots, m$, where $\bar{\v{x}}_{K+1} := \frac{\sum_{k=1}^K \gamma_k \v{x}_{k+1}}{\sum_{k=1}^K \gamma_k}$.
Moreover, if $\gamma_k =\gamma/\sqrt{K}$ for $k=1,\ldots,K$ with $0 < \gamma < P_3^{-1/2}$, then 
\whencolumns{
\begin{align}
\E[ F(\bar{\v{x}}_{K+1})] - \pEstar 
\leq   
	\frac{\eta}{\sqrt{K}},  
\qquad
\E[G^i(\bar{\v{x}}_{K+1})] 
\leq
\frac{\eta}{\sqrt{K}}  
	\label{eq:bound.sqrtK}
\end{align}
}{
\begin{align}
\begin{aligned}
\E[ F(\bar{\v{x}}_{K+1})] - \pEstar 
&\leq   
	\frac{\eta}{\sqrt{K}},  
\\
\E[G^i(\bar{\v{x}}_{K+1})] 
&\leq
\frac{\eta}{\sqrt{K}}  
\end{aligned}
	\label{eq:bound.sqrtK}
\end{align}
}
for $i = 1,\ldots,m$, where $\eta := \frac{P_1 + P_2 \gamma^2}{4\gamma(1-P_3\gamma^2)}$.
\label{thm:exp}
\end{theorem}
A constant step-size of $\eta/\sqrt{K}$ over a fixed number of $K$ iterations yields the $\Ocal(1/\sqrt{K})$ decay rate in the expected distance to optimality and constraint violation of Algorithm \ref{alg:PDSS}. This is indeed order optimal, as implied by Nesterov's celebrated result in \cite[Theorem 3.2.1]{nesterov1998introductory}. 


\begin{remark}
While we present the proof for an i.i.d. sequence of samples, we believe that the result can be extended to the case where $\omega$'s follow a Markov chain with  geometric mixing rate following the technique in \cite{sun2018markov}. For such settings, the expectations in the definition of ${F}, \v{G}$ should be computed with respect to the stationary distribution of the chain. The results will then possibly apply to Markov decision processes with applications in stochastic control.
\end{remark}

Given that the literature on primal-dual subgradient methods is extensive, it is important for us to relate and distinguish Algorithm \ref{alg:PDSS} and Theorem \ref{thm:exp} with prior work.
\rev{Using the Lagrangian in \eqref{eq:P.L},}
Algorithm \ref{alg:PDSS} can be written as 
\begin{equation}
\begin{aligned}
	\v{x}_{k+1} &\coloneqq \proj_{\Xset} [\v{x}_k - \gamma_k \nabla_x \Lcal_\omega (\v{x}_k, \v{z}_k)], 
	\\
	\v{z}_{k+1} &\coloneqq \proj_{{\Rset_+^m}} [\v{z}_k + \gamma_k \nabla_z \Lcal_\omega (\v{x}_{k+1}, \v{z}_k)],
\end{aligned}
\label{eq:proj.alg}
\end{equation}
where $\proj_\Aset$ projects its argument on set $\Aset$. The vectors $\nabla_x \Lcal_\omega$ and $\nabla_z \Lcal_\omega$ are stochastic subgradients of the Lagrangian function with respect to $\v{x}$ and $\v{z}$, respectively. Therefore, Algorithm \ref{alg:PDSS} is a projected stochastic subgradient algorithm that seeks to solve the saddle-point reformulation of $\PE$ \rev{in \eqref{eq:PE.primal}}. Implicit in our algorithm is the assumption that projection on $\Xset$ is computationally easy. Any functional constraints describing $\Xset$ that makes such projection challenging should be included in $\v{G}$.

Closest in spirit to our work on $\PE$ are the papers by Baes et al. in \cite{baes2013randomized}, Yu et al. in \cite{yu2017online}, Xu in \cite{xu2018primal}, and Nedic and Ozdaglar in \cite{nedic2009subgradient}. 
Stochastic mirror-prox algorithm in \cite{baes2013randomized} and projected subgradient method in \cite{nedic2009subgradient} are similar in their updates to ours except in two ways. First, these algorithms in the context of $\PE$ update the dual variable $\v{z}_k$ based on $\v{G}$ or its noisy estimate evaluated at $\v{x}_k$, while we update it based on the estimate at $\v{x}_{k+1}$. Second, both project the dual variable on a compact subset of $\Rset^m_+$ that contains the optimal set of dual multipliers. While authors in \cite{baes2013randomized} assume an a priori set to project on, authors in \cite{nedic2009subgradient} compute such a set from a ``Slater point'' that satisfies $\v{G}(\v{x}) < 0$. \rev{Specifically, Slater's condition guarantees that the set of optimal dual solutions $\Zset_\star$ is bounded (see  \cite[Theorem 1.265]{bonnans2013perturbation}, \cite{hiriart2013convex}). Moreover, a Slater point can be used to construct a compact set that contains $\Zset_\star$, e.g., using  \cite[Lemma 4.1]{nedic2009subgradient}. While one can project dual variables on such a set in each iteration, execution of the algorithm then requires a priori knowledge of such a point. 
We do not assume knowledge of such a point (or any explicit bound on $\Zset_\star$) to execute Algorithm \ref{alg:PDSS}. Rather, our proof provides an explicit bound on the growth of the dual variable sequence for Algorithm \ref{alg:PDSS},} much in line with Xu's analysis in \cite{xu2018primal}. Much to our surprise, a minor modification of using a Gauss-Seidel style dual update as opposed to the popular Jacobi style dual update obviates the need for 
this crucial assumption in the literature for the proofs to work. 
Unfortunately, our Gauss-Seidel style dual update comes at an additional cost of an extra sample required per iteration of the primal-dual algorithm, making the sample complexity double of the iteration complexity. The constant factor of two, however, does not impact the order-wise complexity. We surmise that the additional sample and the Gauss-Seidel update of the dual variable helps to decouple the analysis of the primal and dual updates and points to a possible extension of our result to an asynchronous setting, often useful in engineering applications.

While sharing some parallels, our work has an important difference with that in \cite{xu2018primal}. Xu considers a collection of deterministic constraint functions, i.e., $\v{g}^\omega$ is identical for all $\omega \in \Omega$, and considers a modified augmented Lagrangian function of the form $\tilde{\Lcal}(\v{x}, \v{z}) := F(\v{x}) + \frac{1}{m} \sum_{i=1}^m \varphi_\delta(\v{x}, z^i)$, where
\begin{align}
\varphi_\delta(\v{x}, z^i) :=
\begin{cases}
z^i g^i(\v{x}) + \frac{\delta}{2} [g^i(\v{x})]^2,
& \text{if } \delta g^i(\v{x}) + z^i \geq 0, \\
-\frac{(z^i)^2}{2\delta}, 
& \text{otherwise}
\end{cases}
\end{align}
for $i = 1,\ldots,m$ with a suitable time-varying sequence of $\delta$'s. His algorithm is similar to Algorithm \ref{alg:PDSS} but performs a randomized coordinate update for the dual variable instead of \eqref{eq:zUpdate}. To the best of our knowledge, Xu's analysis in \cite{xu2018primal} with such a Lagrangian function does not directly apply to our setting with stochastic constraints that is crucial for the subsequent analysis of the risk-sensitive problem $\PCVaR$.

Finally, Yu et al.'s work in \cite{yu2017online} provides an analysis of the algorithm that updates its dual variables using 
\begin{alignat}{1}
		\v{z}_{k+1} 
		&\coloneqq 
		\argmax_{\v{z} \in \Rset_+^m}  
			\left\langle \v{v}_k, \v{z} - \v{z}_k \right\rangle 
				- \frac{1}{2\gamma_k} \| \v{z} - \v{z}_k \|^2, \label{eq:zUpdateNeely}
\end{alignat}
where ${v}^i := {g}^i_{\omega_{k}}(\v{x}_{k}) + \langle \nabla {g}^i_{\omega_k}(\v{x}_k) ,\v{x}_{k+1} - \v{x}_k \rangle$ for $i=1,\ldots,m$. In contrast, our $\v{z}$-update in \eqref{eq:zUpdate} samples $\omega_{k+1/2}$ and sets $\v{v}_k := \v{g}_{\omega_{k+1/2}}(\v{x}_{k+1})$ at the already computed point $\v{x}_{k+1}$.
We are able to recover the $\Ocal(1/\sqrt{K})$ decay rate of suboptimality and constraint violation with a proof technique much closer to the classical analysis of subgradient methods in \cite{boyd2006subgradient,nedic2009subgradient}. Unlike \cite{yu2017online}, we provide a clean characterization of the constant $\eta$ in \eqref{eq:bound.sqrtK} that is crucial to study the growth in sample (and iteration) complexity of Algorithm \ref{alg:PDSS} applied to a reformulation of $\PCVaR$.



\subsection{Proof of Theorem \ref{thm:exp}}
\label{sub:proof.thm1}

The proof proceeds in three steps. 
\begin{enumerate}[label=(\alph*)]
\item We establish the following dissipation inequality that consecutive iterates of the algorithm satisfy.
\whencolumns{
\begin{align}
	\begin{aligned}
		&\gamma_k  \E[\Lcal(\v{x}_{k+1}, \v{z})  - \Lcal(\v{x}, \v{z}_k)] 
		+ \frac{1}{2} \E\| \v{x}_{k+1} - \v{x} \|^2 
		+ \frac{1}{2} \E\| \v{z}_{k+1} - \v{z} \|^2 
		\\
		&\quad \le 
		\frac{1}{2} \E\| \v{x}_k - \v{x} \|^2 
		+ \frac{1}{2} \E \| \v{z}_k - \v{z} \|^2 
		+ \frac{1}{4}P_2\gamma_k^2
		+ \frac{1}{4}P_3 \gamma_k^2 \E \| \v{z}_k \|^2
	\end{aligned}
	\label{eq:1step}
\end{align}
}{
\begin{align}
\begin{aligned}
	&\gamma_k  \E[\Lcal(\v{x}_{k+1}, \v{z})  - \Lcal(\v{x}, \v{z}_k)] \\
	&\; + \frac{1}{2} \E\| \v{x}_{k+1} - \v{x} \|^2
	+ \frac{1}{2} \E\| \v{z}_{k+1} - \v{z} \|^2 
	\\
	& \quad \le 
	\frac{1}{2} \E\| \v{x}_k - \v{x} \|^2 
	+ \frac{1}{2} \E \| \v{z}_k - \v{z} \|^2 \\
	&\qquad + \frac{1}{4}P_2\gamma_k^2
	+ \frac{1}{4}P_3 \gamma_k^2 \E \| \v{z}_k \|^2
\end{aligned}
\label{eq:1step}
\end{align}
}
for any $\v{x} \in \Xset$ and $\v{z} \in \Rset_+^m$.


\item Next, we bound $\E\|\v{z}_k\|^2$ generated by our algorithm from above using step (a) as
\begin{align}
\E\|\v{z}_k\|^2 \leq \frac{P_1 + P_2 A_K}{1 - P_3 A_K}
\label{eq:zBound}
\end{align}
for $k = 1, \ldots, K$, where $A_K := \sum_{k=1}^K \gamma_k^2$.

\item We combine the results in steps (a) and (b) to complete the proof.
\end{enumerate}

Define the filtration $\Wcal_1 \subset \Wcal_{1+1/2} \subset \Wcal_2 \subset \ldots $, where $\Wcal_k$ is the $\sigma$-algebra generated by the samples $\omega_1, \ldots \omega_{k-1/2}$ 
for $k$ being multiples of $1/2$, starting from unity. Then, $\{ \v{x}_1, \v{z}_1, \dots, \v{x}_k, \v{z}_k \}$ becomes $\Wcal_k$-measurable, while $\{ \v{x}_1, \v{z}_1, \dots, \v{x}_k, \v{z}_k, \v{x}_{k+1} \}$ is $\Wcal_{k+1/2}$-measurable.


\noindent \emph{$\bullet$ Step (a) -- Proof of \eqref{eq:1step}\whencolumns{:}{}}
We first utilize the $\v{x}$-update in \eqref{eq:xUpdate} to prove
\begin{align}
\begin{aligned}
	&\E[ F(\v{x}_{k+1}) - F(\v{x}) + \v{z}_k^{\T} \v{G}(\v{x}_{k+1}) - \v{z}_k^{\T} \v{G}(\v{x}) | \Wcal_k]
	+ \frac{1}{2\gamma_k} \E[ \| \v{x}_{k+1} - \v{x} \|^2 | \Wcal_k ]
	\\
	&\quad \leq \frac{1}{2\gamma_k} \| \v{x}_k - \v{x} \|^2 
	 + 2\gamma_k(4C_F^2 + \sigma_F^2)
	 + 2\gamma_k m (4 \| \v{C}_G \|^2 + \| \v{\sigma}_G \|^2) \|\v{z}_k\|^2
\end{aligned}
	\label{eq:xTarget}
\end{align}
for all $\v{x} \in \Xset$. 
Then, we utilize the $\v{z}$-update in \eqref{eq:zUpdate} to prove
\begin{align}
\begin{aligned}
 	&\E[ \v{z}^{\T} \v{G}(\v{x}_{k+1}) - \v{z}_k^{\T} \v{G}(\v{x}_{k+1}) | \Wcal_k ]
	+ \frac{1}{2\gamma_k} \E[ \| \v{z}_{k+1} - \v{z} \|^2 | \Wcal_k ]
	\\ &\quad
	 \leq  \frac{1}{2\gamma_k} \| \v{z}_k - \v{z} \|^2
	 +\frac{\gamma_k}{2} \| \v{D}_G \|^2
\end{aligned}
	\label{eq:zTarget}
\end{align}
for all $\v{z} \in \Rset^m_+$. The law of total probability is then applied to the sum of \eqref{eq:xTarget} and \eqref{eq:zTarget} followed by a multiplication by $\gamma_k$ yielding the desired result in \eqref{eq:1step}.

\paragraph{Proof of \eqref{eq:xTarget}\whencolumns{:}{}}

The $\v{x}$-update in \eqref{eq:xUpdate} yields
\whencolumns{
\begin{align}
\left\langle \v{x}_{k+1} - \v{x}, \nabla f_\omega(\v{x}_k) + \sum_{i = 1}^m {z}_k^i \nabla \v{g}_\omega^i(\v{x}_k) + \frac{1}{\gamma_k} (\v{x}_{k+1} - \v{x}_k) \right\rangle \leq 0.
\label{eq:xUpdate.1}
\end{align}
}{
\begin{align}
\left\langle \v{x}_{k+1} - \v{x}, \Delta \v{x} \right\rangle \leq 0,
\label{eq:xUpdate.1}
\end{align}
where $\Delta \v{x}$ is given by
\begin{figure*}
	\Delta \v{x} = \nabla f_\omega(\v{x}_k) + \sum_{i = 1}^m {z}_k^i \nabla \v{g}_\omega^i(\v{x}_k) + \frac{1}{\gamma_k} (\v{x}_{k+1} - \v{x}_k).
\end{figure*}
}
We now simplify the inner product. The product with $\nabla f_\omega(\v{x}_k)$ can be expressed as
\begin{align}
\begin{aligned}
\langle \v{x}_{k+1} - \v{x}, \nabla f_\omega(\v{x}_k) \rangle 
=& \underbrace{\langle \v{x}_{k+1} - \v{x}_k, \nabla F(\v{x}_{k+1}) \rangle}_{\geq F(\v{x}_{k+1}) - F(\v{x}_k) } 
+ \langle \v{x}_k - \v{x}, \nabla f_\omega(\v{x}_k) \rangle \\
&\; - \langle \v{x}_{k+1} - \v{x}_k, \nabla F(\v{x}_{k+1}) - \nabla f_\omega(\v{x}_k) \rangle, 
\end{aligned}
\label{eq:xUpdate.2}
\end{align}
where $\nabla F(\v{x}_{k+1})$ denotes a subgradient of $F$ at $\v{x}_{k+1}$. The inequality for the first term follows from the convexity of $F$. Since $\E[\nabla f_\omega(\v{x}_k) | \Wcal_k] \in \partial F(\v{x}_k)$ from \cite{bertsekas1973stochastic}, the expectation of the second summand on the right hand side (RHS) of \eqref{eq:xUpdate.2} satisfies
\whencolumns{
\begin{align} 
\E [\langle \v{x}_k - \v{x}, \nabla f_\omega(\v{x}_k) \rangle | \Wcal_k] = \langle \v{x}_k - \v{x}, \nabla F(\v{x}_k) \rangle \geq F(\v{x}_k) - F(\v{x}).
\label{eq:xUpdate.T1}
\end{align}
}{
\begin{align} 
\begin{aligned}
\E [\langle \v{x}_k - \v{x}, \nabla f_\omega(\v{x}_k) \rangle | \Wcal_k] 
& = \langle \v{x}_k - \v{x}, \nabla F(\v{x}_k) \rangle \\
& \geq F(\v{x}_k) - F(\v{x}_\star).
\end{aligned}
\label{eq:xUpdate.T1}
\end{align}
}
Taking expectations in \eqref{eq:xUpdate.2}, the above relation implies
\begin{align}
\begin{aligned}
& \E[ \langle \v{x}_{k+1} - \v{x}, \nabla f_\omega(\v{x}_k) \rangle | \Wcal_k  ] \\
& \quad
\geq \E[ F(\v{x}_{k+1}) - F(\v{x}) - \langle \v{x}_{k+1} - \v{x}_k, \nabla F(\v{x}_{k+1}) - \nabla f_\omega(\v{x}_k) \rangle
 | \Wcal_k].
\end{aligned}
\label{eq:xUpdate.1.1}
\end{align}
%
Next, we bound the inner product with the second term on the RHS of \eqref{eq:xUpdate.2}. To that end, utilize the convexity of member functions in $\v{g}_\omega$ and $\v{G}$ along the above lines to infer
\begin{align}
\begin{aligned}
&\sum_{i = 1}^m \E[ \langle \v{x}_{k+1} - \v{x}, {z}_k^i \nabla g_\omega^i(\v{x}_k)\rangle | \Wcal_k ] \\
& \; \geq \sum_{i = 1}^m \E[ z_k^i G^i(\v{x}_{k+1}) - z_k^i G^i(\v{x}) | \Wcal_k ] \\
&\qquad - \E[\langle \v{x}_{k+1} - \v{x}_k, z_k^i \nabla G^i(\v{x}_{k+1}) - z_k^i \nabla g_\omega^i(\v{x}_k) \rangle | \Wcal_k] \\
& \; = \E[ \v{z}_k^{\T} \v{G}(\v{x}_{k+1}) - \v{z}_k^{\T} \v{G}(\v{x}) | \Wcal_k]  \\
&\qquad - \sum_{i = 1}^m \E[\langle \v{x}_{k+1} - \v{x}_k, z_k^i \nabla G^i(\v{x}_{k+1}) -  \v{z}_k^i \nabla g_\omega^i(\v{x}_k) \rangle | \Wcal_k].
\end{aligned}
\label{eq:xUpdate.1.2}
\end{align}
To tackle the inner product with the third term in the RHS of \eqref{eq:xUpdate.1}, we use the identity
\begin{align}
\begin{aligned}
&\left\langle \v{x}_{k+1} - \v{x}, \frac{1}{\gamma_k} (\v{x}_{k+1} - \v{x}_k) \right\rangle \\
&\quad
	= \frac{1}{2\gamma_k} [\| \v{x}_{k+1} - \v{x} \|^2 - \| \v{x}_k - \v{x} \|^2 + \| \v{x}_{k+1} - \v{x}_k \|^2 ].
\end{aligned}
\label{eq:xUpdate.1.3}
\end{align}
The inequalities in \eqref{eq:xUpdate.1.1}, \eqref{eq:xUpdate.1.2}, and the equality in \eqref{eq:xUpdate.1.3} together gives
\begin{align}
\begin{aligned}
	&\E[ F(\v{x}_{k+1}) - F(\v{x}) + \v{z}_k^{\T} \v{G}(\v{x}_{k+1}) - \v{z}_k^{\T} \v{G}(\v{x}) | \Wcal_k]   
	\\
	&\quad 
	- \E[\langle \v{x}_{k+1} - \v{x}_k, \nabla F(\v{x}_{k+1}) - \nabla f_\omega(\v{x}_k) \rangle | \Wcal_k]
	\\
	& \quad 
	- \sum_{i = 1}^m \E[\langle \v{x}_{k+1} - \v{x}_k, {z}_k^i \nabla G^i(\v{x}_{k+1}) - {z}_k^i \nabla g_\omega^i(\v{x}_k) \rangle | \Wcal_k]
	\\
	&\quad
	+ \frac{1}{2\gamma_k} \E[ \| \v{x}_{k+1} - \v{x} \|^2 + \| \v{x}_{k+1} - \v{x}_k \|^2 | \Wcal_k] \\
	&\qquad \leq \frac{1}{2\gamma_k} \| \v{x}_k - \v{x} \|^2.
\end{aligned}
	\label{eq:xUpdate.3}
\end{align}
To simplify the above relation, apply Young's inequality to obtain
\begin{align}
\begin{aligned}
	& \E [\langle \v{x}_{k+1} - \v{x}_k, \nabla F(\v{x}_{k+1}) - \nabla f_\omega(\v{x}_k)  \rangle | \Wcal_k ]
	\\
	&  
	\leq \frac{1}{4\gamma_k} \E [ \| \v{x}_{k+1} - \v{x}_k \|^2 | \Wcal_k]
	+ \gamma_k \E [\| \nabla F(\v{x}_{k+1}) - \nabla f_\omega(\v{x}_k)  \|^2 | \Wcal_k ]
	\\
	&  
	\leq \frac{1}{4\gamma_k} \E[\| \v{x}_{k+1} - \v{x}_k \|^2 | \Wcal_k]
	+ 2\gamma_k \E[ \| \nabla F(\v{x}_{k+1}) 
	\\&\qquad - \E \nabla f_\omega(\v{x}_k) \|^2 + \| \E \nabla f_\omega(\v{x}_k) - \nabla f_\omega(\v{x}_k)  \|^2 | \Wcal_k ].
\end{aligned}
\end{align}
Recall that $\E[ \nabla f_\omega(\v{x}_k)  | \Wcal_k ] \in \partial F(\v{x}_k)$, subgradients of $F$ are bounded and $\nabla f_\omega$ has bounded variance. Therefore, we infer from the above inequality that
\begin{align}
\begin{aligned}
&\E[\langle \v{x}_{k+1} - \v{x}_k, \nabla F(\v{x}_{k+1}) - \nabla f_\omega(\v{x}_k)  \rangle | \Wcal_k ]
\\&\quad
\leq 
\frac{1}{4\gamma_k} \E [ \| \v{x}_{k+1} - \v{x}_k \|^2 | \Wcal_k ]
+ 2\gamma_k (4 C_F^2 + \sigma_F^2 ).
\end{aligned}
\label{eq:Young.1}
\end{align}
Appealing to Young's inequality $m$ times and a similar line of argument as above gives
\whencolumns{
\begin{align}
\begin{aligned}
	& \sum_{i = 1}^m \E[\langle \v{x}_{k+1} - \v{x}_k, {z}_k^i \nabla G^i(\v{x}_{k+1}) - {z}_k^i \nabla g_\omega^i(\v{x}_k) \rangle | \Wcal_k ] 
	\\
	&\quad \leq \frac{1}{4\gamma_k} \E[ \| \v{x}_{k+1} - \v{x}_k \|^2 | \Wcal_k ]
	+ 2\gamma_k m \sum_{i = 1}^m (4 [C_G^i]^2 + [\sigma_G^i]^2) \| \v{z}_k \|^2.
\end{aligned}
	\label{eq:Young.2}
\end{align}
}{
\begin{align}
\begin{aligned}
	& \sum_{i = 1}^m \E[\langle \v{x}_{k+1} - \v{x}_k, {z}_k^i \nabla G^i(\v{x}_{k+1}) - {z}_k^i \nabla g_\omega^i(\v{x}_k) \rangle | \Wcal_k ]  \\
	&\quad \leq \frac{1}{4\gamma_k} \E[ \| \v{x}_{k+1} - \v{x}_k \|^2 | \Wcal_k ] \\
	&\qquad
	+ 2\gamma_k m \sum_{i = 1}^m (4 [C_G^i]^2 + [\sigma_G^i]^2) \| \v{z}_k \|^2.
\end{aligned}
	\label{eq:Young.2}
\end{align}
}
Leveraging the relations in \eqref{eq:Young.1} and \eqref{eq:Young.2} in \eqref{eq:xUpdate.3}, we get
\begin{align}
\begin{aligned}
	&\E[ F(\v{x}_{k+1}) - F(\v{x}) + \v{z}_k^{\T} \v{G}(\v{x}_{k+1}) - \v{z}_k^{\T} \v{G}(\v{x}) | \Wcal_k] 
	\\
	&\;
	+ \frac{1}{2\gamma_k} \E[\| \v{x}_{k+1} - \v{x} \|^2 + \| \v{x}_{k+1} - \v{x}_k \|^2 | \Wcal_k ]
	 \\
	&\quad  \leq \frac{1}{2\gamma_k} \| \v{x}_k - \v{x} \|^2  + \frac{1}{4\gamma_k} \E [ \| \v{x}_{k+1} - \v{x}_k \|^2  | \Wcal_k ] 
	+ 2\gamma_k (4 C_F^2 + \sigma_F^2 )
	\\
	&\qquad
	+ \frac{1}{4\gamma_k} \E [ \| \v{x}_{k+1} - \v{x}_k \|^2 | \Wcal_k]
	+ 2\gamma_k m \sum_{i = 1}^m (4 [C_G^i]^2 + [\sigma_G^i]^2) \|\v{z}_k\|^2,
\end{aligned}
\end{align}
that upon simplification gives \eqref{eq:xTarget}.


\paragraph{Proof of \eqref{eq:zTarget}\whencolumns{:}{}}
From the $\v{z}$-update in \eqref{eq:zUpdate}, we obtain
\begin{align}
	\left\langle \v{z}_{k+1} - \v{z}, -\v{g}_\omega (\v{x}_{k+1}) + \frac{1}{\gamma_k} (\v{z}_{k+1} - \v{z}_k) \right\rangle \le 0
\label{eq:zUpdate.ineq}
\end{align}
for all $\v{z} \geq 0$. Again, we deal with the two summands in the second factor of the inner product  of \eqref{eq:zUpdate.ineq} separately. 
The expectation of the inner product with the first term yields
\begin{align}
\begin{aligned}
&\E[ \langle \v{z}_{k+1} - \v{z}, -\v{g}_\omega (\v{x}_{k+1}) \rangle | \Wcal_{k + {1}/{2}}] 
 \\
& = 
\E[ \langle \v{z}_{k+1} - \v{z}_k, -\v{g}_\omega (\v{x}_{k+1}) \rangle | \Wcal_{k + {1}/{2}} ] 
 + 
\E[ \langle \v{z}_k - \v{z}, -\v{g}_\omega (\v{x}_{k+1}) \rangle | \Wcal_{k + {1}/{2}} ]
 \\
& \geq
- \frac{1}{2\gamma_k}  \E[\| \v{z}_{k+1} - \v{z}_k \|^2 | \Wcal_{k + {1}/{2}}] - \frac{\gamma_k}{2}\E[\|\v{g}_\omega(\v{x}_{k+1})\|^2 | \Wcal_{k + {1}/{2}} ]
\\ &\qquad
+ 
\E[ \langle \v{z}_k - \v{z}, -\v{g}_\omega (\v{x}_{k+1}) \rangle | \Wcal_{k + {1}/{2}} ]
 \\
& \geq
- \frac{1}{2\gamma_k}  \E [ \| \v{z}_{k+1} - \v{z}_k \|^2  | \Wcal_{k + {1}/{2}} ]
- \frac{\gamma_k}{2} \| \v{D}_G \|^2
+ \langle \v{z}_k - \v{z}, -\v{G} (\v{x}_{k+1}) \rangle \footnotemark
 \\
& =
- \frac{1}{2\gamma_k}  \E [ \| \v{z}_{k+1} - \v{z}_k \|^2 | \Wcal_{k + {1}/{2}} ]
- \frac{\gamma_k}{2} \| \v{D}_G \|^2
+ 
 \v{z}^{\T} \v{G} (\v{x}_{k+1}) - \v{z}_k^{\T} \v{G} (\v{x}_{k+1}).
\end{aligned}
\label{eq:zUpdate.2}
\end{align}
\footnotetext{$\E[\v{g}_\omega (\v{x}_{k+1}) | \Wcal_{k + {1}/{2}}] = \v{G}(\v{x}_{k+1})$ requires that we sample $\omega$ once more for the $z$-update.}
In the above derivation, we have utilized Young's inequality and the boundedness of the second moment of $\v{g}_\omega$. Since $\Wcal_k \subset \Wcal_{k + {1}/{2}}$, the law of total probability can be used to condition \eqref{eq:zUpdate.2} on $\Wcal_k$ rather than on $\Wcal_{k + {1}/{2}}$. To simplify the inner product with the second term in \eqref{eq:zUpdate.ineq}, we use the identity
\whencolumns{
\begin{align}
\begin{aligned}
	\left\langle \v{z}_{k+1} - \v{z}, \frac{1}{\gamma_k } (\v{z}_{k+1} - \v{z}_k) \right\rangle
	= \frac{1}{2\gamma_k} [ \| \v{z}_{k+1} - \v{z} \|^2 -  \| \v{z}_k - \v{z} \|^2 +  \| \v{z}_{k+1} - \v{z}_k \|^2].
\end{aligned}
\label{eq:zSquare}
\end{align}
}{
\begin{align}
\begin{aligned}
	&\left\langle \v{z}_{k+1} - \v{z}, \frac{1}{\gamma_k } (\v{z}_{k+1} - \v{z}_k) \right\rangle \\
	&\quad = \frac{1}{2\gamma_k} [ \| \v{z}_{k+1} - \v{z} \|^2 -  \| \v{z}_k - \v{z} \|^2 +  \| \v{z}_{k+1} - \v{z}_k \|^2].
\end{aligned}
\label{eq:zSquare}
\end{align}
}
Utilizing \eqref{eq:zUpdate.2} and \eqref{eq:zSquare} in \eqref{eq:zUpdate.ineq} gives \eqref{eq:zTarget}. Adding \eqref{eq:xTarget} and \eqref{eq:zTarget} followed by a multiplication by $\gamma_k$ yields
\whencolumns{
\begin{align}
\begin{aligned}
	&\gamma_k \E[ \Lcal(\v{x}_{k+1}, \v{z}) - \Lcal(\v{x}, \v{z}_k) | \Wcal_k ] 
	+ \frac{1}{2} \E\left[  \|\v{x}_{k+1} - \v{x} \|^2 +  \| \v{z}_{k+1} - \v{z} \|^2 | \Wcal_k \right] \\
	& \quad \leq \frac{1}{2} \|\v{x}_k - \v{x} \|^2 + \frac{1}{2} \| \v{z}_k - \v{z} \|^2 + \frac{1}{4} P_2 \gamma_k^2 + \frac{1}{4} P_3 \gamma_k^2 \| \v{z}_k \|^2.
\end{aligned}
\label{eq:almostMartingale}
\end{align}
}{
\begin{align}
\begin{aligned}
	&\gamma_k \E[ \Lcal(\v{x}_{k+1}, \v{z}) - \Lcal(\v{x}, \v{z}_k) | \Wcal_k ] 
	\\
	&\;
	+ \frac{1}{2} \E\left[  \|\v{x}_{k+1} - \v{x} \|^2 +  \| \v{z}_{k+1} - \v{z} \|^2 | \Wcal_k \right] \\
	& \quad \leq \frac{1}{2} \|\v{x}_k - \v{x} \|^2 + \frac{1}{2} \| \v{z}_k - \v{z} \|^2 + \frac{1}{4} P_2 \gamma_k^2 + \frac{1}{4} P_3 \gamma_k^2 \| \v{z}_k \|^2.
\end{aligned}
\label{eq:almostMartingale}
\end{align}
}
Taking the expectation and applying the law of total probability completes the proof of \eqref{eq:1step}.

\noindent \emph{$\bullet$ Step (b) -- Proof of \eqref{eq:zBound}\whencolumns{:}{}}
Plugging $(\v{x}, \v{z}) = (\v{x}_\star, \v{z}_\star)$ in the inequality for the one-step update in \eqref{eq:1step} and summing it over $k=1, \ldots, \kappa$ for $\kappa \leq K$ gives
\begin{align}
\begin{aligned}
	&\sum_{k=1}^\kappa \gamma_k \underbrace{\E[\Lcal(\v{x}_{k+1}, \v{z}_\star)  - \Lcal(\v{x}_\star, \v{z}_k)] }_{\geq 0 \text{ from } \eqref{eq:saddle}} 
	+ \frac{1}{2} \sum_{k=1}^\kappa \left[ \E\| \v{x}_{k+1} - \v{x}_\star \|^2 + \E\| \v{z}_{k+1} - \v{z}_\star \|^2 \right]
	 \\
	&\quad \le 
	\frac{1}{2} \sum_{k=1}^\kappa \left[ \E\| \v{x}_k - \v{x}_\star \|^2 + \E \| \v{z}_k - \v{z}_\star \|^2 \right]
	+ \frac{1}{4} P_2 \sum_{k=1}^\kappa  \gamma_k^2 
	+ \frac{1}{4} P_3  \sum_{k=1}^\kappa \gamma_k^2 \E\|\v{z}_k\|^2
\end{aligned}
\label{eq:1step.z*.1}
\end{align}
for $\kappa = 1,\ldots,K$. The above then yields
\whencolumns{
\begin{align}
\begin{aligned}
	&\underbrace{ \E\| \v{x}_{\kappa+1} - \v{x}_\star \|^2}_{\geq 0} 
	+  {\E\| \v{z}_{\kappa+1} - \v{z}_\star \|^2} 
	\\
	&\quad \le 
	 \| \v{x}_1 - \v{x}_\star \|^2 
	+  \| \v{z}_1 - \v{z}_\star \|^2 
	+ \frac{1}{2} P_2 \sum_{k=1}^\kappa  \gamma_k^2 
	+ \frac{1}{2} P_3  \sum_{k=1}^\kappa \gamma_k^2 \E\|\v{z}_k\|^2.
\end{aligned}
\label{eq:1step.z*.2}
\end{align}
}{
\begin{align}
\begin{aligned}
	&\underbrace{ \E\| \v{x}_{\kappa+1} - \v{x}_\star \|^2}_{\geq 0} 
	+  {\E\| \v{z}_{\kappa+1} - \v{z}_\star \|^2} \\
	&\quad \le 
	 \| \v{x}_1 - \v{x}_\star \|^2 
	+  \| \v{z}_1 - \v{z}_\star \|^2 
	\\
	&\qquad
	+ \frac{1}{2} P_2 \sum_{k=1}^\kappa  \gamma_k^2 
	+ \frac{1}{2} P_3  \sum_{k=1}^\kappa \gamma_k^2 \E\|\v{z}_k\|^2.
\end{aligned}
\label{eq:1step.z*.2}
\end{align}
}
Notice that $ 2 \E\|\v{z}_{\kappa+1} - \v{z}_\star\|^2 +  2\|\v{z}_\star\|^2 \geq \E\| \v{z}_{\kappa+1}\|^2$. This inequality and $\v{z}_1 = 0$ in \eqref{eq:1step.z*.2} gives
\whencolumns{
\begin{align}
\begin{aligned}
	\E\|\v{z}_{\kappa+1}\|^2
	& \le 2 \|\v{x}_{1} - \v{x}_\star\|^2 +  4 \|\v{z}_\star\|^2 + P_2 \sum_{k=1}^\kappa \gamma_k^2 + P_3 \sum_{k=1}^\kappa \gamma_k^2 \E\|\v{z}_k\|^2 \\
	& \le P_1 + P_2 \sum_{k=1}^\kappa \gamma_k^2 + P_3 \sum_{k=1}^\kappa \gamma_k^2 \E\|\v{z}_k\|^2.
\end{aligned}
\label{eq:1step.z*.3}
\end{align}
}{
\begin{align}
\begin{aligned}
	\E\|\v{z}_{\kappa+1}\|^2
	& \le 2 \|\v{x}_{1} - \v{x}_\star\|^2 +  4 \|\v{z}_\star\|^2
	\\
	&\quad + P_2 \sum_{k=1}^\kappa \gamma_k^2 + P_3 \sum_{k=1}^\kappa \gamma_k^2 \E\|\v{z}_k\|^2 \notag \\
	& \le P_1 + P_2 \sum_{k=1}^\kappa \gamma_k^2 + P_3 \sum_{k=1}^\kappa \gamma_k^2 \E\|\v{z}_k\|^2.
\end{aligned}
\label{eq:1step.z*.3}
\end{align}
}
We argue the bound on $\E\|\v{z}_k\|^2$ for $k=1,\ldots,K$ inductively. Since $\v{z}_1 = 0$, the base case trivially holds. Assume that the bound holds for $k = 1,\ldots,\kappa$ for $\kappa < K$. With the notation $A_K = \sum_{k=1}^K \gamma_k^2$, the relation in \eqref{eq:1step.z*.3} implies 
\begin{align}
\E\|\v{z}_{\kappa+1}\|^2 
& \leq P_1 + P_2 \sum_{k=1}^{\kappa} \gamma_k^2 + P_3 \sum_{k=1}^{\kappa} \gamma_k^2  \frac{P_1 + P_2 A_K}{1 - P_3 A_K} 
\notag \\
& \leq P_1 + P_2  A_K + P_3  \frac{P_1 + P_2 A_K}{1 - P_3 A_K} A_K
\notag \\
& =
\frac{P_1 + P_2 A_K}{1 - P_3  A_K},
\end{align}
completing the proof of step (b).


\noindent \emph{$\bullet$ Step (c) -- Combining steps (a) and (b) to prove Theorem \ref{thm:exp}\whencolumns{:}{}}
For any $\v{z} \geq 0$, the inequality in \eqref{eq:1step} with $\v{x} = \v{x}_\star$ from step (a) summed over $k=1,\ldots, K$ gives 
\whencolumns{
	\begin{align}
	\begin{aligned}
		&\sum_{k=1}^K \gamma_k \E[\Lcal(\v{x}_{k+1}, \v{z})  - \Lcal(\v{x}_\star, \v{z}_k)] 
		+ \frac{1}{2} \sum_{k=1}^K  \left[ \E\| \v{x}_{k+1} - \v{x}_\star \|^2 + \E\| \v{z}_{k+1} - \v{z} \|^2 \right]
		 \\
		& \le 
		\frac{1}{2} \sum_{k=1}^K \left[ \E\| \v{x}_k - \v{x}_\star \|^2 
		+ \E \| \v{z}_k - \v{z} \|^2 \right]
		+ \frac{1}{4} P_2\sum_{k=1}^K  \gamma_k^2 
		+ \frac{1}{4} P_3 \sum_{k=1}^K  \gamma_k^2 \E\|\v{z}_k\|^2.
	\end{aligned}
	\label{eq:1step.z}
	\end{align}
}{
	\begin{align}
	\begin{aligned}
		&\sum_{k=1}^K \gamma_k \E[\Lcal(\v{x}_{k+1}, \v{z})  - \Lcal(\v{x}_\star, \v{z}_k)] 
		\\
		&\;
		+ \frac{1}{2} \sum_{k=1}^K  \E\| \v{x}_{k+1} - \v{x}_\star \|^2 
		+ \frac{1}{2} \sum_{k=1}^K  \E\| \v{z}_{k+1} - \v{z} \|^2 
		 \\
		&\quad \le 
		\frac{1}{2} \sum_{k=1}^K  \E\| \v{x}_k - \v{x}_\star \|^2 
		+ \frac{1}{2} \sum_{k=1}^K  \E \| \v{z}_k - \v{z} \|^2 
		\\
		&\qquad
		+ \frac{1}{4} P_2\sum_{k=1}^K  \gamma_k^2 
		+ \frac{1}{4} P_3 \sum_{k=1}^K  \gamma_k^2 \E\|\v{z}_k\|^2.
	\end{aligned}
	\label{eq:1step.z}
	\end{align}
}
Using $\v{z}_1 = 0$ and an appeal to \rev{the saddle point property of $(\v{x}_\star, \v{z}_\star)$} yields
\whencolumns{
\begin{alignat}{1}
	&\sum_{k=1}^K \gamma_k \E[\Lcal(\v{x}_{k+1}, \v{z})  - \Lcal(\v{x}_\star, \v{z}_\star)] 
	+ \frac{1}{2} \E\| \v{x}_{K+1} - \v{x}_\star \|^2 
	+ \frac{1}{2} \E\| \v{z}_{K+1} - \v{z} \|^2 
	 \notag \\
	&\quad \le 
	\frac{1}{2} \E\| \v{x}_1 - \v{x}_\star \|^2 
	+ \frac{1}{4} P_2\sum_{k=1}^K  \gamma_k^2 
	+ \frac{1}{4} P_3 \sum_{k=1}^K  \gamma_k^2 \E\|\v{z}_k\|^2
	+ \frac{1}{2} \|\v{z}\|^2
	\notag \\
	&\quad \le
	\frac{1}{4} P_1 - \| \bone + \v{z}_\star\|^2
	+ \frac{1}{4} P_2 A_K		
	+ \frac{1}{4} P_3 \sum_{k=1}^K  \gamma_k^2 \frac{P_1 + P_2A_K}{1 - P_3 A_K}
	+ \frac{1}{2} \|\v{z}\|^2
	\notag \\
	&\quad =
	\frac{1}{4} P_1
	+ \frac{1}{4} P_2 A_K		
	+ \frac{1}{4} P_3 A_K \frac{P_1 + P_2A_K}{1 - P_3 A_K}
	+ \frac{1}{2} \|\v{z}\|^2  - \|\bone + \v{z}_\star\|^2
	\notag \\
	&\quad =
	\frac{1}{4} \left( \frac{P_1 + P_2 A_K}{1 - P_3 A_K} \right)
	+ \frac{1}{2} \|\v{z}\|^2  - \|\bone +\v{z}_\star\|^2.
\label{eq:1step.z.2}
\end{alignat}
}{
\begin{alignat}{1}
	&\sum_{k=1}^K \gamma_k \E[\Lcal(\v{x}_{k+1}, \v{z})  - \Lcal(\v{x}_\star, \v{z}_\star)] 
	\notag\\
	&\;
	+ \frac{1}{2} \E\| \v{x}_{K+1} - \v{x}_\star \|^2 
	+ \frac{1}{2} \E\| \v{z}_{K+1} - \v{z} \|^2 
	 \notag \\
	&\quad \le 
	\frac{1}{2} \E\| \v{x}_1 - \v{x}_\star \|^2 
	+ \frac{1}{2} \|\v{z}\|^2
	\notag \\
	&\qquad 
	+ \frac{1}{4} P_2\sum_{k=1}^K  \gamma_k^2 
	+ \frac{1}{4} P_3 \sum_{k=1}^K  \gamma_k^2 \E\|\v{z}_k\|^2
	\notag \\
	&\quad \le
	\frac{1}{4} P_1 - \| \bone + \v{z}_\star\|^2 
	+ \frac{1}{2} \|\v{z}\|^2
	\notag \\
	&\qquad
	+ \frac{1}{4} P_2 A_K		
	+ \frac{1}{4} P_3 \sum_{k=1}^K  \gamma_k^2 \frac{P_1 + P_2A_K}{1 - P_3 A_K}
	\notag \\
	&\quad =
	\frac{1}{4} P_1
	+ \frac{1}{4} P_2 A_K		
	+ \frac{1}{4} P_3 A_K \frac{P_1 + P_2A_K}{1 - P_3 A_K}
	\notag \\
	&\qquad
	+ \frac{1}{2} \|\v{z}\|^2  - \|\bone + \v{z}_\star\|^2
	\notag \\
	&\quad =
	\frac{1}{4} \left( \frac{P_1 + P_2 A_K}{1 - P_3 A_K} \right)
	+ \frac{1}{2} \|\v{z}\|^2  - \|\bone +\v{z}_\star\|^2.
\label{eq:1step.z.2}
\end{alignat}
}
In deriving the above inequality, we have utilized the bound on $\E\|\v{z}_k\|^2$ from step (b) and the definition of $P_1$ and $A_K$. 
To further simplify the above inequality, notice that \rev{the saddle point property of $(\v{x}_\star, \v{z}_\star)$ in \eqref{eq:saddle} yields
\begin{align}
    F(\v{x}_\star) = \Lcal(\v{x}_\star, 0) \leq \Lcal(\v{x}_\star, \v{z}_\star) = F(\v{x}_\star) + \v{z}_\star^\T \v{G}(\v{x}_\star),
\end{align} 
which implies $\v{z}_\star^\T \v{G}(\v{x}_\star) \geq 0$. However, the saddle point theorem guarantees that $\v{x}_\star$ is an optimizer of $\PE$, meaning that $\v{x}_\star$ is feasible and $\v{G}(\v{x}_\star) \leq 0$, implying $\v{z}_\star^\T \v{G}(\v{x}_\star) \leq 0$ as $\v{z}_\star \in \Rset^m_+$. Taken together, we infer
\begin{align}
    \v{z}_\star^\T \v{G}(\v{x}_\star) = 0 \implies \Lcal(\v{x}_\star, \v{z}_\star) = F(\v{x}_\star) = \pEstar.
    \label{eq:comp.slack}
\end{align}
}
Since $\Lcal(\v{x}, \v{z})$ is convex in $\v{x}$, Jensen's inequality \rev{and \eqref{eq:comp.slack} implies}
\whencolumns{
\begin{align}
\sum_{k=1}^K \gamma_k \E[\Lcal(\v{x}_{k+1}, \v{z})  - \Lcal(\v{x}_\star, \v{z}_\star)] 
\geq
\left( {\sum_{k=1}^K \gamma_k} \right) \E[\Lcal(\bar{\v{x}}_{K+1}, \v{z}) - \rev{\pEstar}],
\label{eq:JensenL}
\end{align}
}{
\begin{align}
\begin{aligned}
&\sum_{k=1}^K \gamma_k \E[\Lcal(\v{x}_{k+1}, \v{z})  - \Lcal(\v{x}_\star, \v{z}_\star)] 
\\
&\quad
\geq
\left( {\sum_{k=1}^K \gamma_k} \right) \E[\Lcal(\bar{\v{x}}_{K+1}, \v{z}) - \rev{\pEstar}],
\end{aligned}
\label{eq:JensenL}
\end{align}
}
where recall that $\bar{\v{x}}_{K+1}$ is the $\gamma$-weighted average of the iterates. Utilizing \eqref{eq:JensenL} in \eqref{eq:1step.z.2}, we get
\whencolumns{
\begin{align}
\begin{aligned}
	\left(\sum_{k=1}^K \gamma_k \right) \E[ \Lcal(\bar{\v{x}}_{K+1}, \v{z}) - \rev{\pEstar}]
	& \leq
	\frac{1}{4} \left( \frac{P_1 + P_2 A_K}{1 - P_3 A_K} \right)
	+ \frac{1}{2} \|\v{z}\|^2  - \|\bone + \v{z}_\star\|^2.
\end{aligned}
\label{eq:1step.z.3}
\end{align}
}{
\begin{align}
\begin{aligned}
	&\left(\sum_{k=1}^K \gamma_k \right) \E[ \Lcal(\bar{\v{x}}_{K+1}, \v{z}) - \rev{\pEstar}]
	\\
	&\quad \leq
	\frac{1}{4} \left( \frac{P_1 + P_2 A_K}{1 - P_3 A_K} \right)
	+ \frac{1}{2} \|\v{z}\|^2  - \|\bone + \v{z}_\star\|^2.
\end{aligned}
\label{eq:1step.z.3}
\end{align}
}
The above relation defines a bound on $\E[\Lcal(\bar{\v{x}}_{K+1}, \v{z})]$ for every $\v{z} \geq 0$.
Choosing $\v{z} = 0$ and noting $\| \bone + \v{z}_\star \|^2 \ge 0$, we get the bound on expected suboptimality in \eqref{eq:FBound}. To derive the bound on expected constraint violation in \eqref{eq:GBound}, notice that \rev{the saddle point property in \eqref{eq:saddle}  and \eqref{eq:comp.slack} imply}
\begin{align}
\begin{aligned}
&\E[\Lcal(\bar{\v{x}}_{K+1}, \bone^i+\v{z}_\star) - \rev{\pEstar}] 
\\&\quad = \E[\Lcal(\bar{\v{x}}_{K+1}, \v{z}_\star) - \Lcal(\v{x}_\star, \v{z}_\star)] 
+ \E\left[[\bone^i]^{\T} \v{G}(\bar{\v{x}}_{K+1})\right]
\\&\quad \geq \E[G^i(\bar{\v{x}}_{K+1})],
\end{aligned}
\label{eq:Gi.bound}
\end{align}
where $\bone^i \in \Rset^m$ is a vector of all zeros except the $i$-th entry that is unity. Choosing $\v{z} = \bone^i + \v{z}_\star$ in \eqref{eq:1step.z.3} and the observation in \eqref{eq:Gi.bound} then gives
\begin{align}
\E[G^i(\bar{\v{x}}_{K+1})]
&\; \leq 
	\frac{1}{4 \sum_{k=1}^K \gamma_k }   
		\left( 
			\frac{P_1 + P_2 A_K}{1 - P_3 A_K} + 2 \|\bone^i + \v{z}_\star\|^2 - 4 \| \bone + \v{z}_\star\|^2
		 \right) \notag \\
&\; \leq 
	\frac{1}{4 \sum_{k=1}^K \gamma_k }  
		\left( \frac{P_1 + P_2 A_K}{1 - P_3 A_K} \right)
\end{align}
for each $i=1,\ldots,m$.
This completes the proof of \eqref{eq:GBound}. The bounds in \eqref{eq:bound.sqrtK} are immediate from that in \eqref{eq:FBound}--\eqref{eq:GBound}. This completes the proof of Theorem \ref{thm:exp}.


\rev{\begin{remark}
The bound in \eqref{eq:GBound} can be sharpened to 
\begin{align}
    \sum_{i=1}^m  \E\left[ {G}^i(\bar{\v{x}}_{K+1})\right]^+ \leq 	\frac{1}{4 \sum_{k=1}^K \gamma_k }  
	\left( 
		\frac{P_1 + P_2 \sum_{k=1}^K \gamma_k^2}{1 - P_3 \sum_{k=1}^K \gamma_k^2}
	\right)
	\label{eq:rev.violation}
\end{align}
using $\v{z}$ defined by $
    \v{z}^i \coloneqq \v{z}^i_\star + \indc_{\{{G}^i(\bar{\v{x}}_{K+1}) > 0 \}}
$ for $i = 1, \ldots, m$
in \eqref{eq:1step.z.3}. Here, $\indc_{\{A\}}$ is the indicator function, evaluating to $1$ if $A$ holds and $0$ otherwise. This improved bound was suggested to us by an anonymous reviewer. Notice that  \eqref{eq:rev.violation} suggests a much tighter bound on the expected constraint violation per constraint than \eqref{eq:GBound} when $m$ is large.
\end{remark}
}

In what follows, we offer insights into two specific aspects of our proof. 
First, we present our conjecture on where the Gauss-Seidel nature of our dual update obtained with an extra sample helps us  circumvent the need for an a priori bound on the dual variable. Notice that our dual update allows us to derive the third line of \eqref{eq:zUpdate.2} that ultimately yields the term $-\v{z}_k^{\T} \v{G}(\v{x}_{k+1})$ in \eqref{eq:zTarget}. This term conveniently disappears when \eqref{eq:zTarget} is added to the inequality in \eqref{eq:xTarget} obtained from the primal update. We conjecture that this cancellation made possible by our dual update makes the theoretical analysis particularly easy. We anticipate that the classical Jacobi-style dual iteration derived with one sample shared within the primal and the dual steps will not lead to said cancellation and yield a term of the form $\v{z}_k^{\T} \left[ \v{G}(\v{x}_{k+1}) - \v{G}(\v{x}_{k}) \right]$.
Bounding the growth of such a term might prove challenging without an available bound on $\|\v{z}_k\|$ and will likely require a different argument. A detailed comparison between the proof techniques of the Jacobi and the Gauss-Seidel updates is left for future endeavors.

Second, we comment on the presence of a dimensionless constant $\bone$ in $P_1$ together with $\v{z}_\star$. 
We use the inequality in \eqref{eq:1step} to establish \eqref{eq:1step.z.3} that is valid at all $\v{z} \geq 0$. Inspired by arguments in \cite{xu2018primal}, we then utilize \eqref{eq:1step.z.3} not only at the dual iterate $\v{z}_k$--that is often the case with many prior analyses--but also at $\v{z}=0$ and $\v{z} = \bone^i + \v{z}_\star$. Specifically, the nature of the Lagrangian function $\Lcal(\v{x}, \v{z})$ in $\v{z}$ permits us to relate these evaluations at $\v{z}=0$ and $\v{z} = \bone^i + \v{z}_\star$ to the extents of suboptimality and constraint violation, respectively, using 
\begin{align}
	\mathcal{L}(\v{x}, 0) = F(\v{x}), \quad \mathcal{L}(\v{x}, \bone^i + \v{z}) = \mathcal{L}(\v{x}, \v{z}) + G^i(\v{x}).
\end{align}
The deliberate inclusion of $\| {\bone+\v{z}_\star} \|^2$ in the constant $P_1$ aids in drowning the effect of the term $\frac{1}{2}\|{\v{z}}\|^2$ in \eqref{eq:1step.z.3} evaluated at $\v{z} = \bone^i + \v{z}_\star$ when deriving the bound on the extent of constraint violation, without impacting the same when evaluated at $\v{z}=0$, used in deriving the bound on the extent of suboptimality.


\subsection{Optimal step size selection}
\label{sub:stepsize}
We exploit the bounds in Theorem \ref{thm:exp} to select a step size that minimizes the iteration count to reach an $\ve$-approximately feasible and optimal solution to $\PE$ and solve\footnote{The integrality of $K$ is ignored for notational convenience.} 
\begin{align}
\begin{aligned}
& \underset{K, \ \gamma > 0}{\text{minimize}} && \ K, \\
& \text{subject to } && \   
\frac{\eta}{\sqrt{K}} = \frac{P_1 + P_2 \gamma^2}{4 \gamma \sqrt{K}(1 - P_3 \gamma^2)} \le \ve, \ \ P_3 \gamma^2 < 1.
\end{aligned}
\label{eq:step.minK}
\end{align}
The following characterization of optimal step sizes and the resulting iteration count from Proposition \ref{prop:gamma.K.opt} will prove useful in studying the growth in iteration complexity in solving $\PCVaR$ with the risk-aversion parameters $\alpha, \v{\beta}$ in the following section.

\begin{proposition}
	For any $\ve > 0$, the optimal solution of \eqref{eq:step.minK} satisfies
\begin{align}
\begin{aligned}
	\gamma_\star^2 = \frac{2 P_3^{-1}}{2 + y + \sqrt{y^2 + 8y}}, 
\; \;
	K_\star =  \frac{(P_1 + P_2 \gamma_\star^2)^2}{16  \gamma_\star^2 (1 - P_3 \gamma_\star^2)^2 \ve^2},
\end{aligned}
\label{eq:gamma.K.opt}
\end{align}
where $y = 1 + \frac{P_2}{P_1 P_3}$.
\label{prop:gamma.K.opt}
\end{proposition}

\begin{proof}
It is evident from \eqref{eq:gamma.K.opt} that $\gamma_\star^2 < P_3^{-1}$. Then, it suffices to show that $\gamma_\star$ from \eqref{eq:gamma.K.opt} minimizes 
\begin{align}
    \sqrt{K} =  \frac{P_1 + P_2 \gamma^2}{4  \gamma (1 - P_3 \gamma^2) \ve}
\end{align}
over $\gamma > 0$. To that end, notice that
\begin{alignat}{1}
	\frac{d}{d\gamma} \left( \frac{P_1 + P_2 \gamma^2}{\gamma (1 - P_3 \gamma^2)}\right) 
	&= \frac{P_2 P_3 \gamma^4 + (P_2 + 3 P_1 P_3) \gamma^2 - P_1}{\gamma^2 (1 - P_3 \gamma^2)^2}.\end{alignat}
The above derivative is negative at $\gamma = 0^+$ and vanishes only at $\gamma_\star$ over positive values of $\gamma$, certifying it as the global minimizer.
\end{proof}

Parameter $P_1$ is generally not known a priori. However, it is often possible to  bound it from above. One can calculate $\gamma_\star$ and $K_\star$ using \eqref{eq:gamma.K.opt}, replacing $P_1$ with its overestimate. Notice that
\begin{align}
\frac{d K_\star}{d P_1} := \frac{\partial K_\star}{\partial P_1} + \frac{\partial K_\star}{\partial \gamma_\star} \frac{d \gamma_\star}{d y} \frac{d y}{d P_1}.    
\end{align}
It is straightforward to verify that  $\frac{\partial K_\star}{\partial P_1} > 0$, 
$\frac{d y}{d P_1} \le 0$, and $\frac{\partial \gamma_\star}{\partial y} \le 0$, and hence, overestimating $P_1$ results in a smaller $\gamma_\star$. Finally, $\frac{\partial K_\star}{\partial \gamma} > 0$ for $\gamma > \gamma_\star$, implying that $K_\star$ calculated with an overestimate of $P_1$ is larger than the optimal iteration count--the computational burden we must bear for not knowing $P_1$. Our algorithm does require knowledge of $P_3$ to implement the algorithm, that in turn depend only on the nature of the functions defining the constraints and not a primal-dual optimizer.


\subsection{Asymptotic almost sure convergence with decaying step-sizes}
\label{sub:almost_sure_convergence}
Subgradient methods are often studied with decaying non-summable square-summable step sizes, for which they converge to an optimizer in the unconstrained setting. The result holds even for distributed variants and for mirror descent methods (see \cite{doan2018convergence}). Establishing convergence of Algorithm \ref{alg:PDSS} to a primal-dual optimizer of $\PE$ is much more challenging without assumptions of strong convexity in the objective. With such step-sizes, we provide the following result to guarantee the stability of our algorithm, which is reminiscent of \cite[Theorem 4]{nedic2014stochastic}.
\begin{proposition}
	Suppose Assumption \ref{assumptions} holds and $\{\gamma_k\}_{k=1}^\infty$ is a non-summable square-summable nonnegative sequence, i.e., $\sum_{k=1}^\infty \gamma_k = \infty, \sum_{k=1}^\infty \gamma_k^2 < \infty$. Then, $(\v{x}_k, \v{z}_k)$ generated by Algorithm \ref{alg:PDSS} remains bounded and
	$\lim_{k\to\infty} \Lcal(\v{x}_k, \v{z}_\star) - \Lcal(\v{x}_\star, \v{z}_k) = 0$ 
almost surely.
\label{prop:iterates}
\end{proposition}
This `gap' function $\Lcal(\v{x}, \v{z}_\star) - \Lcal(\v{x}_\star, \v{z})$ looks notoriously similar to the duality gap at $(\v{x}, \v{z})$, but is not the same. We are unaware of any results on asymptotic almost sure convergence of primal-dual first-order algorithms to an optimizer for constrained convex programs with convex, but not necessarily strongly convex, objectives. A recent result in \cite{yamashita2020passivity} establishes such a convergence in primal-dual dynamics in continuous time; our attempts at leveraging discretizations of the same have yet proven unsuccessful.

The proof of Proposition \ref{prop:iterates} takes advantage of the one-step update in \eqref{eq:1step}, that makes it amenable to the well-studied almost supermartingale convergence result by Robbins and Siegmund in \cite[Theorem 1]{robbins1971convergence}.
\begin{theoremnn}[Convergence of almost supermartingales]
	Let  $m_k, n_k, r_k, s_k$ be $\Fcal_k$-measurable finite nonnegative random variables, where $\Fcal_1 \subseteq \Fcal_2 \subseteq \ldots$ describes a filtration. If $\sum_{k=1}^\infty s_k < \infty$, $\sum_{k=1}^\infty r_k < \infty$, and
	\begin{align}
		\E[m_{k+1} | \Fcal_k] \le m_k(1 + s_k) + r_k - n_k, 
	\label{eq:rs}
	\end{align}
	then $\lim_{k \to \infty} m_k$ exists and is finite and $\sum_{k = 1}^\infty n_k < \infty$ almost surely.
\end{theoremnn}

\begin{proof}[Proof of Proposition \ref{prop:iterates}]
Using notation from the proof of Theorem \ref{thm:exp},  \eqref{eq:xTarget} and \eqref{eq:zTarget} together yields
\whencolumns{
\begin{align}
	\begin{aligned}
		&\gamma_k  \E[\Lcal(\v{x}_{k+1}, \v{z}_\star)  - \Lcal(\v{x}, \v{z}_k) | \Wcal_k] 
		+ \frac{1}{2} \E\left[ \| \v{x}_{k+1} - \v{x}_\star \|^2
		+  \| \v{z}_{k+1} - \v{z}_\star \|^2  | \Wcal_k \right]
		\\
		&\quad \le 
		\frac{1}{2} \left[ \| \v{x}_k - \v{x}_\star \|^2 
		+ \| \v{z}_k - \v{z}_\star \|^2 \right]
		+ \frac{1}{4}P_2\gamma_k^2
		+ \frac{1}{4}P_3 \gamma_k^2 \| \v{z}_k \|^2.
	\end{aligned}
	\label{eq:1step.as}
\end{align}
}{
\begin{align}
\begin{aligned}
	&\gamma_k  \E[\Lcal(\v{x}_{k+1}, \v{z}_\star)  - \Lcal(\v{x}_\star, \v{z}_k) | \Wcal_k] \\
	&\; + \frac{1}{2} \E\left[ \| \v{x}_{k+1} - \v{x}_\star \|^2
		+  \| \v{z}_{k+1} - \v{z}_\star \|^2  | \Wcal_k \right]
	\\
	& \quad \le 
	\frac{1}{2} \left[ \| \v{x}_k - \v{x}_\star \|^2 
	+ \| \v{z}_k - \v{z}_\star \|^2 \right] \\
	&\qquad + \frac{1}{4}P_2\gamma_k^2
	+ \frac{1}{4}P_3 \gamma_k^2 \| \v{z}_k \|^2.
\end{aligned}
\label{eq:1step.as}
\end{align}
}
We utilize the above to derive a similar inequality replacing $\E[\Lcal(\v{x}_{k+1}, \v{z}_\star) | \Wcal_k]$ with $\Lcal(\v{x}_{k}, \v{z}_\star)$ by bounding the difference between them. Then, we apply the almost supermartingale convergence theorem to the result to conclude the proof. To bound said difference, the convexity of $\Lcal$ in $\v{x}$ and Young's inequality together imply
\begin{align}
\begin{aligned}
	&\Lcal(\v{x}_k, \v{z}_\star) - \E[\Lcal(\v{x}_{k+1}, \v{z}_\star)|\Wcal_k]
	\\&\quad\le
	\langle \nabla \Lcal(\v{x}_{k+1}, \v{z}_\star), \v{x}_{k} - \v{x}_{k+1} \rangle
	\\ &\quad
	\le
	\frac{\gamma_k}{2}\E[\| {\nabla}_x \Lcal(\v{x}_{k+1}, \v{z}_\star) \|^2 | \Wcal_k]
	+ \frac{1}{2\gamma_k} \E[\| \v{x}_k - \v{x}_{k+1} \|^2 | \Wcal_k],
\end{aligned}
\label{eq:1step.Lnorm}
\end{align}
where ${\nabla}_x \Lcal$ denotes a subgradient of $\Lcal$ w.r.t. $\v{x}$. 
To further bound the RHS of \eqref{eq:1step.Lnorm}, Assumption \ref{assumptions} allows us to deduce
\whencolumns{
\begin{align}
\begin{aligned}
	\| \nabla \Lcal(\v{x}, \v{z}_\star) \|^2
	&\le 
	2 \| \nabla F(\v{x}) \|^2 + 2m\sum_{i = 1}^m\| z_\star^i \nabla G^i(\v{x}) \|^2
	\\
	&\le
	2 C_F^2 + 2m \| \v{z}_\star \|^2 \| \v{C}_G \|^2 
	\\
	& := 2Q_1.
\end{aligned}
\label{eq:normBound.L}
\end{align}
}{
\begin{align}
\hspace{-0.1in}
\begin{aligned}
	\| {\nabla}_x \Lcal(\v{x}, \v{z}_\star) \|^2
	&=
	\left\| \nabla F(\v{x}) + \sum_{i = 1}^m z_\star^i \nabla G^i(\v{x}) \right\|^2
	\\
	&\le 
	2 \| \nabla F(\v{x}) \|^2 + 2m\sum_{i = 1}^m\| z_\star^i \nabla G^i(\v{x}) \|^2
	\\
	&\le
	2 C_F^2 + 2m \| \v{z}_\star \|^2 \| \v{C}_G \|^2 
	\\
	& := 2Q_1.
\end{aligned}
\label{eq:normBound.L}
\end{align}
}
for any $\v{x} \in \Xset$.
Furthermore, the $\v{x}$-update in \eqref{eq:proj.alg} and the non-expansive nature of the projection operator yield
\begin{align}
\begin{aligned}
&\frac{1}{\gamma_k^2}\E[ \| \v{x}_{k+1} - \v{x}_k \|^2 | \Wcal_k]
\\&\quad \le 
\E\left[ \left\| \nabla f_\omega(\v{x}_k) + \sum_{i = 1}^m z_k^i \nabla g_\omega^i(\v{x}_k) \right\|^2 | \Wcal_k \right]
\\
&\quad \leq
2 \E[ \| \nabla f_\omega(\v{x}_k) \|^2 | \Wcal_k] + 2m\sum_{i = 1}^m  \E\left[ (z_k^i)^2 \|\nabla g_\omega^i(\v{x}_k) \|^2 | \Wcal_k\right].
\end{aligned}
\label{eq:xk.xk1}
\end{align}
From Assumption \ref{assumptions}, we get
\begin{align}
\E[ \| \nabla f_\omega(\v{x}_k) \|^2 | \Wcal_k] 
& \leq 
2 \E[ \| \nabla f_\omega(\v{x}_k) - \E \nabla f_\omega(\v{x}_k)\|^2 + \| \E \nabla f_\omega(\v{x}_k) \|^2 | \Wcal_k] \notag
\\ &
\leq 2 \sigma_F^2 + 2 C_F^2,
\end{align}
and along similar lines
\begin{align}
\sum_{i = 1}^m  \E\left[ (z_k^i)^2 \|\nabla g_\omega^i(\v{x}_k) \|^2 | \Wcal_k\right]
\leq 2 (\| \v{\sigma}_G\|^2 + \|\v{C}_G\|^2) \| \v{z}_k \|^2,
\end{align}
that together in \eqref{eq:xk.xk1} yield
\whencolumns{
\begin{align}
\frac{1}{\gamma_k^2}\E[ \| \v{x}_{k+1} - \v{x}_k \|^2 | \Wcal_k]
\leq \underbrace{4 (\sigma_F^2 + C_F^2)}_{:=2Q_2} 
+ \underbrace{4m(\| \v{\sigma}_G\|^2 + \|\v{C}_G\|^2)}_{:=2Q_3} \| \v{z}_k \|^2.
\end{align}
}{
\begin{align*}
& \frac{1}{\gamma_k^2}\E[ \| \v{x}_{k+1} - \v{x}_k \|^2 | \Wcal_k]
\\
& \quad \leq \underbrace{4 (\sigma_F^2 + C_F^2)}_{:=2Q_2} 
+ \underbrace{4m(\| \v{\sigma}_G\|^2 + \|\v{C}_G\|^2)}_{:=2Q_3} \| \v{z}_k \|^2.
\end{align*}
}
Combining the above with \eqref{eq:normBound.L} in \eqref{eq:1step.Lnorm} gives
\whencolumns{
\begin{align}
\begin{aligned}
	\gamma_k\left( \Lcal(\v{x}_k, \v{z}_\star) - \E[\Lcal(\v{x}_{k+1}, \v{z}_\star)|\Wcal_k] \right)
	\le
	\gamma_k^2 (Q_1 + Q_2 + Q_3  \| \v{z}_k \|^2).
\end{aligned}
\label{eq:Lbound}
\end{align}
}{
\begin{align}
\begin{aligned}
	&\gamma_k\left( \Lcal(\v{x}_k, \v{z}_\star) - \E[\Lcal(\v{x}_{k+1}, \v{z}_\star)|\Wcal_k] \right)
	\\
	& \quad \le
	\gamma_k^2 (Q_1 + Q_2 + Q_3  \| \v{z}_k \|^2).
\end{aligned}
\label{eq:Lbound}
\end{align}
}
Adding \eqref{eq:Lbound} to \eqref{eq:1step.as} and simplifying, we obtain
\whencolumns{
\begin{align}
	\begin{aligned}
		&\frac{1}{2} \E\left[ \| \v{x}_{k+1} - \v{x}_\star \|^2
		+  \| \v{z}_{k+1} - \v{z}_\star \|^2  | \Wcal_k \right]
		\\
		&\quad \le 
		\frac{1}{2} \left[ \| \v{x}_k - \v{x}_\star \|^2 
		+ \| \v{z}_k - \v{z}_\star \|^2 \right]
		- \gamma_k  [\Lcal(\v{x}_{k}, \v{z}_\star) - \Lcal(\v{x}_\star, \v{z}_k)]
		\\
		&\qquad
		+ \gamma_k^2 \left( \frac{1}{4}P_2 + Q_1 + Q_2\right)
		+  \gamma_k^2 \left( \frac{1}{4}P_3 + Q_3 \right) \| \v{z}_k \|^2.
	\end{aligned}
	\label{eq:martingale}
\end{align}
}{
\begin{align}
	&\frac{1}{2} \E\left[ \| \v{x}_{k+1} - \v{x}_\star \|^2
		+  \| \v{z}_{k+1} - \v{z}_\star \|^2  | \Wcal_k \right]
	\\
	& \quad \le 
	\frac{1}{2} \left[ \| \v{x}_k - \v{x}_\star \|^2 
	+ \| \v{z}_k - \v{z}_\star \|^2 \right] \\
	&\qquad 
	- \gamma_k  [\Lcal(\v{x}_{k}, \v{z}_\star) - \Lcal(\v{x}_\star, \v{z}_k) ]  \\
	&\qquad
		+ \gamma_k^2 \left( \frac{1}{4}P_2 + Q_1 + Q_2\right)
		+  \gamma_k^2 \left( \frac{1}{4}P_3 + Q_3 \right) \| \v{z}_k \|^2.
\end{align}
}
The above inequality with
\begin{align}\| \v{z}_k \|^2 \leq 2\| \v{x}_k - \v{x}_\star \|^2 + 2 \| \v{z}_k - \v{z}_\star \|^2 + 2 \| \v{z}_\star \|^2
\end{align} 
becomes \eqref{eq:rs}, where
\whencolumns{
\begin{gather}
\begin{gathered}
	m_k = \frac{1}{2} \E\| \v{x}_k - \v{x}_\star \|^2 
	+ \frac{1}{2} \E \| \v{z}_k - \v{z}_\star \|^2, \quad
	n_k = \gamma_k  [\Lcal(\v{x}_{k}, \v{z}_\star) - \Lcal(\v{x}_\star, \v{z}_k) ], \\
	r_k = \gamma_k^2 \left[ \frac{1}{4}P_2 + Q_1 + Q_2 + \left( \frac{1}{2}P_3 + 2 Q_3 \right) \| \v{z}_\star \|^2  \right], \quad
	s_k = \gamma_k^2 \left( \frac{1}{2}P_3 + 2 Q_3 \right).
	\end{gathered}
	\end{gather}
}{
\begin{align}
	m_k &= \frac{1}{2} \| \v{x}_k - \v{x}_\star \|^2 
	+ \frac{1}{2} \| \v{z}_k - \v{z}_\star \|^2, \\
	n_k &= \gamma_k  [\Lcal(\v{x}_{k}, \v{z}_\star) - \Lcal(\v{x}_\star, \v{z}_k) ], \\
	r_k &= \gamma_k^2 \left[ \frac{1}{4}P_2 + Q_1 + Q_2 + \left( \frac{1}{2}P_3 + 2 Q_3 \right) \| \v{z}_\star \|^2  \right], \\
	s_k &= \gamma_k^2 \left( \frac{1}{2}P_3 + 2 Q_3 \right).
\end{align}
}
Each term is nonnegative, owing to \eqref{eq:saddle}, and $\v{\gamma}$ defines a square summable sequence. Applying \cite[Theorem 1]{robbins1971convergence}, $m_k$ converges to a constant and $\sum_{k = 1}^\infty n_k < \infty$. The latter combined with the non-summability of $\v{\gamma}$ implies the result.
\end{proof}



\section{Algorithm for \texorpdfstring{$\PCVaR$}{PCVaR} and its analysis}
\label{sec:cvar}
We now devote our attention to solving $\PCVaR$ via a primal-dual algorithm. To do so, we reformulate it as an instance of $\PE$ and utilize Algorithm \ref{alg:PDSS} to solve that reformulation \rev{with constant step-sizes} under a stronger set of assumptions given below.
\begin{assumption}\ 
\begin{enumerate}[label=(\alph*)]
\item Subgradients of $F$ and $\v{G}$ are bounded, i.e.,
\whencolumns{
$\| \nabla f_\omega(\v{x}) \| \leq C_F$ and $\|\nabla g_\omega^i(\v{x}) \| \leq C_G^i$ almost surely for all $\v{x} \in \Xset$.
}{
\begin{gather*}
\begin{aligned}	
\| \nabla f_\omega(\v{x}) \| \leq C_F, \quad 
\|\nabla g_\omega^i(\v{x}) \| \leq C_G^i
\end{aligned}
\end{gather*}
almost surely for all $\v{x} \in \Xset$.
}

\item $\v{g}_\omega(x)$ is bounded, i.e., $\| g^i_\omega(\v{x}) \| \le D_G^i$ for all $\v{x} \in \Xset$, almost surely.


\item \rev{$\Lcal$ for $\PCVaR$ admits a saddle point $(\v{x}_\star, \v{z}_\star) \in \Xset \times \Rset^m_+$.}\footnote{Lemma \ref{lem:saddle} provides  sufficient conditions for the existence of such a saddle point.}

\end{enumerate}
\label{assumptions.cvar}
\end{assumption}

Using the variational characterization \eqref{eq:CVaR.def} of $\CVaR$, rewrite $\PCVaR$ as  
\whencolumns{
\begin{align}
\begin{aligned}
& \underset{\v{x} \in \Xset}{\text{minimize}} && \  \min_{u^0 \in \Rset} \E[ \psi^f_\omega(\v{x}, u^0; \alpha)], \\ & \text{subject to } && \   
\min_{u^i \in \Rset} \E[\psi^{g^i}_\omega(\v{x}, u^i; \beta^i)] \leq 0, \quad i=1,\ldots, m,
\end{aligned}
\label{eq:prob.P.cvar}
\end{align}
}{
\begin{align}
\begin{aligned}
\PCVaR : \quad
& \underset{\v{x} \in \Xset}{\text{minimize}} && \  \min_{u^0 \in \Rset} \E[ \psi^f_\omega(\v{x}, u^0; \alpha)], \\
& \text{subject to } && \   
\min_{u^i \in \Rset} \E[\psi^{g^i}_\omega(\v{x}, u^i; \beta^i)] \leq 0, \\
&&& \ \ \text{for each } i=1,\ldots, m,
\end{aligned}
\label{eq:prob.P.cvar}
\end{align}
}
where 
${\psi^h_\omega(\v{x}, u; \delta)} \coloneqq {u + \frac{1}{1-\delta}[h_\omega(\v{x}) - u]^+}$
for any collection of convex functions $h_\omega: \Rset^n \to \Rset$, $\omega \in \Omega$.  Coupled with Assumption \ref{assumptions.cvar}, \rev{we will show that we can bound} $|u^i| \leq D_G^i$ for each $i = 1, \dots, m$\footnote{$\CVaR$ of any random variable can only vary between the mean and the maximum value that random variable can take.}, that allows us to rewrite $\PCVaR$ as
\whencolumns{
\begin{align}
\begin{aligned}
\PE' : \quad
& \underset{\substack{\v{x} \in \Xset, u^0 \in \Rset, \\ | \v{u} | \le \v{D}_G}}{\text{minimize}} && \ \E[ \psi^f_\omega(\v{x}, u^0; \alpha)], 
\\ 
& \text{subject to } && \  
\E[\psi^{g^i}_\omega(\v{x}, u^i; \beta^i)] \leq 0, \quad \text{for each } i=1,\ldots, m,
\end{aligned}
\label{eq:prob.P.expp}
\end{align}
}{
\begin{align}
\begin{aligned}
\PE' : \quad
& \underset{\substack{\v{x} \in \Xset, u^0 \in \Rset, \\ | \v{u} | \le \v{D}_G}}{\text{minimize}} && \ \E[ \psi^f_\omega(\v{x}, u^0; \alpha)], 
\\ 
& \text{subject to } && \  
\E[\psi^{g^i}_\omega(\v{x}, u^i; \beta^i)] \leq 0, \\
&&& \ \text{for each } i=1,\ldots, m,
\end{aligned}
\label{eq:prob.P.expp}
\end{align}
}
where $| \cdot |$ denotes the element-wise absolute value. 
Call the optimal value of $\PCVaR$ as $p^\CVaR_\star$ in the sequel. 
%


\begin{theorem}[Convergence result for $\PCVaR$]
	Suppose Assumption \ref{assumptions.cvar} holds. The iterates generated by Algorithm \ref{alg:PDSS} on $\PE'$ for $\PCVaR$ with parameters $\alpha, \v{\beta}$ satisfy
	\begin{alignat}{1}
		\E[\CVaR_{\alpha}(f_\omega(\bar{\v{x}}_{K+1}))] - p^\CVaR_\star &\le \frac{\eta(\alpha, \v{\beta})}{\sqrt{K}}, \label{eq:FBound.cvar} \\
		\E[\CVaR_{\beta^i}(g_\omega^i(\bar{\v{x}}_{K+1}))] &\le \frac{\eta(\alpha, \v{\beta})}{\sqrt{K}} \label{eq:GBound.cvar}
	\end{alignat}
for $i = 1, \ldots, m$ with step sizes $\gamma_k = \gamma / \sqrt{K}$ for $k = 1, \dots, K$ with $0 < \gamma < {P_3^{-1/2}(\alpha, \v{\beta})}$, where $\eta(\alpha, \v{\beta}) \coloneqq \frac{P_1 + \gamma^2 P_2(\alpha, \v{\beta})}{4\gamma(1 - \gamma^2 P_3(\alpha, \v{\beta}))}$ and
	\whencolumns{
	\begin{alignat}{1}
	\hspace{-0.1in}
	\begin{aligned}
		P_2(\alpha, \v{\beta}) &\coloneqq \frac{16 (C_F^2 + 1)}{(1 - \alpha)^2} + 2 \vnorm{ \diag(\bone + \v{\beta}) \diag(\bone - \v{\beta})^{-1} \v{D}_G }^2, \\
		P_3(\alpha, \v{\beta}) 
		&\coloneqq 
		16 m \vnorm{ 
		\begin{pmatrix}
		\diag(\bone - \v{\beta})^{-1} \v{C}_G \\ 
		\diag(\bone - \v{\beta})^{-1} \bone 
		\end{pmatrix}
		}^2.
	\end{aligned}
	\label{eq:P.cvar}
	\end{alignat}
	}{
	\begin{alignat}{1}
	\hspace{-0.1in}
	\begin{aligned}
		P_2(\alpha, \v{\beta}) &\coloneqq \frac{16 (C_F^2 + 1)}{(1 - \alpha)^2} \\ &\qquad + 2 \vnorm{ \diag(\bone + \v{\beta}) \diag(\bone - \v{\beta})^{-1} \v{D}_G }^2, \\
		P_3(\alpha, \v{\beta}) 
		&\coloneqq 
		16 m \vnorm{ 
		\begin{pmatrix}
		\diag(\bone - \v{\beta})^{-1} \v{C}_G \\ 
		\diag(\bone - \v{\beta})^{-1} \bone 
		\end{pmatrix}
		}^2.
	\end{aligned}
	\label{eq:P.cvar}
	\end{alignat}
	}
	\label{thm:cvar}
\end{theorem}


\begin{proof}
We prove the result in the following steps.
\begin{enumerate}[label=(\alph*)]
	\item Under Assumption \ref{assumptions.cvar}, we revise $P_2$ and $P_3$ in Theorem \ref{thm:exp} for $\PE$.
	\item We show that if $f_\omega, \v{g}_\omega$ satisfy Assumption \ref{assumptions.cvar}, then $\psi^f_\omega$ and $\psi^{g^i}_\omega, i=1,\ldots,m$ satisfy Assumption \ref{assumptions.cvar}, but with different bounds on the gradients and function values. Leveraging these bounds, we obtain $P_2(\alpha, \v{\beta})$ and $P_3(\alpha, \v{\beta})$ for $\PE'$ using step (a).
	
	\item \rev{Using Assumption \ref{assumptions.cvar}, we prove that the Lagrangian function $\Lcal':\Xset \times \Rset \times \Uset \times \Rset^m_+ \to \Rset$ defined as
	\begin{align}
	    \Lcal'(\v{x}, u^0, \v{u}, \v{z}) \coloneqq \E[ \psi^f_\omega(\v{x}, u^0; \alpha)] + \sum_{i=1}^m z^i \E[ \psi^f_\omega(\v{x}, u^0; \alpha)]
	    \label{eq:L_PE'}
	\end{align} 
	admits a saddle point in $\Xset \times \Rset \times \Uset \times \Rset^m_+$, where $\Uset \coloneqq \left\{\v{u} \in \Rset^m \mid | \v{u}| \leq \v{D}_G\right\}$.}

	\item We then apply Theorem \ref{thm:exp} with  $P_2(\alpha, \v{\beta})$ and $P_3(\alpha, \v{\beta})$  from step (b) on $\PE'$ to derive the bounds in \eqref{eq:FBound.cvar} and \eqref{eq:GBound.cvar}.
\end{enumerate}


\noindent \emph{$\bullet$ Step (a) -- Revising Theorem \ref{thm:exp} with Assumption \ref{assumptions.cvar}\whencolumns{:}{}}
Recall that in the derivation of \eqref{eq:Young.1} in the proof of Theorem \ref{thm:exp}, Assumption \ref{assumptions} yields
\begin{align}
    \| \nabla F(\v{x}_{k+1}) - \nabla f_\omega(\v{x}_k) \|^2 \leq 2(4C_F^2 + \sigma_F^2).\end{align} 
Assumption \ref{assumptions.cvar} allows us to bound the same by $4C_F^2$, yielding $P_2 = 16 C_F^2 + 2\| \v{D}_G \|^2$. 
Along the same lines, we get $P_3 = 16 m \| \v{C}_G \|^2$.


\noindent \emph{$\bullet$ Step (b) -- Deriving properties of $\psi_\omega$\whencolumns{:}{}}
Consider the stochastic subgradient of $\psi^f_\omega(\v{x}, t; \alpha)$ given by
\begin{alignat}{1}
	\nabla \psi_\omega^f(\v{x}, u; \alpha) &= \begin{pmatrix}
		\frac{1}{1 - \alpha} \nabla f_\omega(\v{x}) \indc_{\{ f_\omega(\v{x}) \ge u\}} \\
		1 - \frac{1}{1 - \alpha} \indc_{\{ f_\omega(\v{x}) \ge u\}}
	\end{pmatrix},
\end{alignat}
where $\indc_{\{\cdot \}}$ is the indicator function.
Recall that $\| \nabla f_\omega(\v{x}) \| \le C_F$ for all $\v{x} \in \Xset$ almost surely. Therefore, we have
\begin{align}
\begin{aligned}
	\| \nabla \psi_\omega^f(\v{x}, u; \alpha) \|^2
	& =
	\left\| \frac{1}{1 - \alpha} \nabla f_\omega(\v{x}) \indc_{\{ f_\omega(\v{x}) \ge u\}} \right\|^2
	+ \left\| 1 - \frac{1}{1 - \alpha} \indc_{\{ f_\omega(\v{x}) \ge u\}} \right\|^2 
	\\
	& \le
	\frac{C_F^2 + 1}{(1 - \alpha)^2}.
\end{aligned}
	\label{eq:cvar.subBound.f}
\end{align}
Proceeding similarly, we obtain
\begin{align}
\vnorm{\nabla \psi^{g^i}_\omega(\v{x}, u^i; \beta^i)}^2 
\leq \frac{[C_G^i]^2 + 1}{(1-\beta^{i})^2}.
\label{eq:cvar.subBound.g}
\end{align}
%
We also have
\whencolumns{
\begin{align}
\begin{aligned}
	\| \psi^{g^i}_\omega(\v{x}, u^i; \beta^i) \| 
	= \vnorm{  \max\left\lbrace \frac{g^i_\omega(\v{x}) - \beta^i u^i}{1 - \beta^i},  u^i \right\rbrace } 
	\le \frac{1 + \beta^i}{1 - \beta^i} D_G^i.
\end{aligned}
	\label{eq:cvar.normBound}
\end{align}
}{
\begin{align}
\begin{aligned}
	\| \psi^{g^i}_\omega(\v{x}, u^i; \beta^i) \| 
	&
	= \vnorm{ u^i + \frac{1}{1 - \beta^i}[g^i_\omega(\v{x}) - u^i]^+}
	\\
	&= \vnorm{  \max\left\lbrace \frac{g^i_\omega(\v{x}) - \beta^i u^i}{1 - \beta^i},  u^i \right\rbrace }
	\\
	&\le \frac{1 + \beta^i}{1 - \beta^i} D_G^i.
\end{aligned}
	\label{eq:cvar.normBound}
\end{align}
}
Then, \eqref{eq:P.cvar} follows from step (a) using \eqref{eq:cvar.subBound.f}, \eqref{eq:cvar.subBound.g}, and \eqref{eq:cvar.normBound}.


\noindent \rev{\emph{$\bullet$ Step (c) -- Showing that $\Lcal'$ admits a saddle point\whencolumns{:}{}}
According to \cite[Theorem 10]{rockafellar2002optimization}, the minimizers of $\E[ \psi^f_\omega(\v{x}, u^0; \alpha)]$ over $u^0$ define a nonempty closed bounded interval (possibly a singleton). Thus, we have
\begin{align}
    F(\v{x}) = \E[ \psi^f_\omega(\v{x}, u^0(\v{x}); \alpha)]
    \label{eq:FCVaR.rep}
\end{align} 
for some $u^0(\v{x})\in \Rset$ for each $\v{x} \in \Xset$. Similarly, we infer
\begin{align}
    G^i(\v{x}) = \E[ \psi^{g^i}_\omega(\v{x}, u^i(\v{x}); \beta^i)]
    \label{eq:GCVaR.rep}
\end{align} 
for some $u^i(\v{x})\in \Rset$ for each $\v{x} \in \Xset$. Moreover, for all $u^i > D_G^i$, we have 
\begin{align}
    \E[ \psi^{g^i}_\omega(\v{x}, u^i; \beta^i)] = u^i,
\end{align}
 and for $u^i < -D_G^i$, we have
\begin{align}
    \E[ \psi^{g^i}_\omega(\v{x}, u^i; \beta^i)] 
    =  \frac{1}{1-\beta^i} \left(\E[g^i_\omega(\v{x}) 
    -{\beta^i}u^i \right).
\end{align}
Thus $\E[ \psi^{g^i}_\omega(\v{x}, u^i; \beta^i)]$ is non-increasing in $u^i$ below $-D_G^i$ and increasing in it beyond $D_G^i$. Hence, at least one among the minimizers of $\E[ \psi^{g^i}_\omega(\v{x}, u^i; \beta^i)]$ must lie in $[-D_G^i, D_G^i]$. In the sequel, let  $u^i(\v{x})$ refer to such a minimizer.}

\rev{Consider a saddle point $(\v{x}_\star, \v{z}_\star) \in \Xset \times \Rset^m_+$ of $\PCVaR$. We argue that $(\v{x}_\star, u^0(\v{x}_\star), \v{u}(\v{x}_\star), \v{z}_\star)$ is a saddle point of $\Lcal'$. From the definitions of $\Lcal$, $\Lcal'$,  \eqref{eq:FCVaR.rep}, \eqref{eq:GCVaR.rep}, and the saddle point property of $(\v{x}_\star, \v{z}_\star)$, we obtain
\begin{align}
\begin{aligned}
\Lcal'(\v{x}_\star, u^0(\v{x}_\star), \v{u}(\v{x}_\star), \v{z}_\star)
&=
\Lcal(\v{x}_\star, \v{z}_\star)
\\
&\leq 
\Lcal(\v{x}, \v{z}_\star)
\\
&=
\E[ \psi^{f}_\omega(\v{x}, u^0(\v{x}); \alpha)] + \sum_{i=1}^m z^i_\star \E[ \psi^{g^i}_\omega(\v{x}, u^i(\v{x}); \beta^i)]
\\
&\leq \Lcal'(\v{x}, u^0, \v{u}, \v{z}_\star)
\end{aligned}
\label{eq:saddle.cvar.1}
\end{align}
for all $(\v{x}, u^0, \v{u}) \in \Xset \times \Rset \times \Uset$. Also, for all $\v{z} \in \Rset^m_+$, we have 
\begin{align}
\begin{aligned}
\Lcal'(\v{x}_\star, u^0(\v{x}_\star), \v{u}(\v{x}_\star), \v{z})
= \Lcal(\v{x}_\star, \v{z})
\leq \Lcal(\v{x}_\star, \v{z}_\star)
=
\Lcal'(\v{x}_\star, u^0(\v{x}_\star), \v{u}(\v{x}_\star), \v{z}_\star).
\end{aligned}
\label{eq:saddle.cvar.2}
\end{align}}


\noindent \emph{$\bullet$ Step (d) -- Proof of \eqref{eq:FBound.cvar} and \eqref{eq:GBound.cvar}\whencolumns{:}{}}
\rev{By the saddle point theorem and \eqref{eq:saddle.cvar.1}, we have $\Lcal(\v{x}_\star, \v{z}_\star) = p_\star^\CVaR$, that also equals the optimal value of $\PE'$.} 
Applying Theorem \ref{thm:exp} with revised $P_2$ and $P_3$ from step (b) to $\PE'$ for which $\v{x}_0, \dots, \v{x}_{K+1}$ and $u^0_0, \dots, u^0_{K+1}$ are $\Wcal_{K+1/2}$-measurable, we obtain
\whencolumns{
\begin{alignat}{1}
\E[ \CVaR_\alpha(f_\omega(\bar{\v{x}}_{K+1}))  ] 
&= \E\left[ 
\min_{u^0\in\Rset} 
\E[ \psi^f_\omega(\bar{\v{x}}_{K+1}, u^0; \alpha) | \Wcal_{K+1/2}]
\right] 
\notag \\
&\leq \E\left[  
\E[ \psi^f_\omega(\bar{\v{x}}_{K+1}, \bar{u}^0_{K+1}; \alpha) | \Wcal_{K+1/2}]
\right] 
\notag \\
&= \E\left[ \psi^f_\omega(\bar{\v{x}}_{K+1}, \bar{u}^0_{K+1}; \alpha) 
\right] 
\notag \\
& \leq p^\CVaR_\star + \frac{\eta(\alpha, \v{\beta})}{\sqrt{K}}.
\end{alignat}
}{
\begin{alignat}{1}
&\E[ \CVaR_\alpha(f_\omega(\bar{\v{x}}_{K+1}))  ] 
\notag \\
&\quad = \E\left[ 
\min_{u^0\in\Rset} 
\E[ \psi^f_\omega(\bar{\v{x}}_{K+1}, u^0; \alpha) | \Wcal_{K+1/2}]
\right] 
\notag \\
&\quad \leq \E\left[  
\E[ \psi^f_\omega(\bar{\v{x}}_{K+1}, \bar{u}^0_{K+1}; \alpha) | \Wcal_{K+1/2}]
\right] 
\notag \\
&\quad = \E\left[ \psi^f_\omega(\bar{\v{x}}_{K+1}, \bar{u}^0_{K+1}; \alpha) 
\right] 
\notag \\
&\quad \leq p^\CVaR_\star + \frac{\eta(\alpha, \v{\beta})}{\sqrt{K}}.
\end{alignat}
}
%
Following a similar argument for $i=1,\ldots,m$, we get 
\whencolumns{
\begin{alignat}{1}
\E\left[ \CVaR_{\beta^i}(g^i_\omega(\bar{\v{x}}_{K+1}))\right] 
= \E\left[ \min_{u^i\in\Rset} \E[\psi^{g^i}_\omega(\bar{\v{x}}_{K+1}, u^i; \beta^i)  | \Wcal_{K+1/2}] \right] 
\leq \frac{\eta(\alpha, \v{\beta})}{\sqrt{K}},
\end{alignat}
}{
\begin{alignat}{1}
\E\left[ \CVaR_{\beta^i}(g^i_\omega(\bar{\v{x}}_{K+1}))\right] 
\notag \\
 = \E\left[ \min_{u^i\in\Rset} \E[\psi^{g^i}_\omega(\bar{\v{x}}_{K+1}, u^i; \beta^i)  | \Wcal_{K+1/2}] \right] 
 \leq  \frac{\eta(\alpha, \v{\beta})}{\sqrt{K}},
\end{alignat}
}
completing the proof.
\end{proof}

Our proof architecture generalizes to problems with other risk measures as long as that measure preserves convexity of $f_\omega, \v{g}_\omega$, admits a variational characterization as in \eqref{eq:CVaR.def}, and a subgradient for this modified objective can be easily computed and remains bounded over $\Xset$. We restrict our attention to $\CVaR$ to keep the exposition concrete.

\rev{Opposed to sample average approximation (SAA) algorithms, we neither compute nor estimate $F(\v{x}) =\CVaR[f_\omega(\v{x})]$, $\v{G}(\v{x}) = \CVaR[\v{g}_\omega(\v{x})]$ for any given decision $\v{x}$ to run the algorithm. Yet, our analysis provides guarantees on the same at $\bar{\v{x}}_{K+1}$ in expectation. If one needs to compute $F$ at any decision variable, e.g., at $\bar{\v{x}}_{K+1}$, one can employ the variational characterization in \eqref{eq:CVaR.def}. Such evaluation requires additional computational effort. 
Notice that Theorem \ref{thm:cvar} does not relate $F(\bar{\v{x}}_{K+1})$ to $p_\star^\CVaR$ in an almost sure sense; it only relates the two in expectation according to \eqref{eq:FBound.cvar}, where the expectation is evaluated with respect to the stochastic sample path.}

$\CVaR$ of a random variable depends on the tail of its distribution. The higher the risk aversion, the further into the tail one needs to look, generally requiring more samples.
\rev{Even if we do not explicitly compute the tail-dependent CVaR relevant to the objective or the constraints, it is natural to expect our sample complexity to grow with risk aversion, which the following result confirms.}
\begin{proposition}
	Suppose Assumption \ref{assumptions.cvar} holds. For an $\ve$-approximately feasible and optimal solution of $\PCVaR$ with risk aversion parameters $\alpha, \v{\beta}$ using Algorithm \ref{alg:PDSS} on $\PE'$, then $\gamma_\star(\alpha, \v{\beta})$ and  $K_\star(\alpha, \v{\beta})$ from Proposition \ref{prop:gamma.K.opt}, respectively decreases and increases with both $\alpha$ and $\v{\beta}$.
\end{proposition}

\begin{proof} 
	We borrow the notation from Proposition \ref{prop:gamma.K.opt} and tackle the variation with $\alpha$ and $\v{\beta}$ separately.
	
\noindent \emph{$\bullet$ Variation with $\alpha$\whencolumns{:}{}}
	$P_2$ increases with $\alpha$, implying $\gamma_\star$ decreases with $\alpha$ because $\frac{d \gamma_\star^2}{d y} \le 0$ and $\frac{d y}{d P_2} \ge 0$. Furthermore, using $\frac{\partial K_\star}{\partial \gamma_\star} < 0$ for $\gamma < \gamma_\star$ and $\frac{\partial K_\star}{\partial P_2} \geq 0$ in
	\begin{equation}
		\frac{d K_\star}{d P_2} = \frac{\partial K_\star}{\partial P_2} + \frac{\partial K_\star}{\partial \gamma_\star}\frac{d \gamma_\star}{d P_2}
	\end{equation}
we infer that $K_\star$ increases with $\alpha$.

\noindent \emph{$\bullet$ Variation with $\beta^i$\whencolumns{:}{}} Both $P_2$ and $P_3$ increase with $\beta^i$ and
	\begin{align}
	\frac{d \gamma_\star^2}{d {\beta^i}} 
	= \frac{\partial \gamma_\star^2}{\partial P_2} \frac{d P_2}{d {\beta^i}}
	+ \frac{\partial \gamma_\star^2}{\partial P_3} \frac{d P_3}{d {\beta^i}}.
	\end{align}
Following an argument similar to that for the variation with $\alpha$, the first term on the RHS of the above equation can be shown to be nonpositive. Next, we show that the second term is nonpositive to conclude that $\gamma_\star$ decreases with $\beta^i$, where we use $\frac{d P_3}{d {\beta^i}} \ge 0$. Utilizing $\frac{P_2}{P_1 P_3} = y - 1$, we infer
\whencolumns{
\begin{align}
\begin{aligned}
\frac{\partial \gamma_\star^2}{\partial P_3} 
&= 
-\frac{2}{P_3^2(2 + y + \sqrt{y^2 + 8y})} + \frac{\partial \gamma_\star^2}{\partial y} \frac{\partial y}{\partial P_3} \\
&= 
-\frac{2}{P_3^2(2 + y + \sqrt{y^2 + 8y})}  
+ 2\frac{4 + y + \sqrt{y^2 + 8y}}{P_3 \sqrt{y^2 + 8y} (2 + y + \sqrt{y^2 + 8y})^2} \frac{P_2}{P_1 P_3^2}  \\
&=
-2\frac{{ 5y + 4 +  3 \sqrt{y^2 + 8y} }}{P_3^2 \sqrt{y^2 + 8y} (2 + y + \sqrt{y^2 + 8y})^2} \\
& \leq 0.
\end{aligned}
\end{align}
}{
\begin{align}
\begin{aligned}
\frac{\partial \gamma_\star^2}{\partial P_3} 
&= 
-\frac{2}{P_3^2(2 + y + \sqrt{y^2 + 8y})} + \frac{\partial \gamma_\star^2}{\partial y} \frac{\partial y}{\partial P_3}  \\
&= 
-\frac{2}{P_3^2(2 + y + \sqrt{y^2 + 8y})}  
 \\
&\ \quad
+ 2\frac{4 + y + \sqrt{y^2 + 8y}}{P_3 \sqrt{y^2 + 8y} (2 + y + \sqrt{y^2 + 8y})^2} \frac{P_2}{P_1 P_3^2}  \\
&=
-2\frac{{ 5y + 4 +  3 \sqrt{y^2 + 8y} }}{P_3^2 \sqrt{y^2 + 8y} (2 + y + \sqrt{y^2 + 8y})^2} \\
& \leq 0.
\end{aligned}
\end{align}
}
To characterize the variation of $K_\star$, notice that
	\begin{gather}
		\frac{d K_\star}{d \beta^i}
		=
		\frac{\partial K_\star}{\partial P_2}\frac{\partial P_2}{\partial \beta^i} 		+ \frac{\partial K_\star}{\partial P_3}\frac{\partial P_3}{\partial \beta^i} .
	\end{gather}
Again, the first term on the RHS of the above relation is nonnegative, owing to an argument similar to that used for the variation of $K_\star$ with $\alpha$. We show $\frac{\partial K_\star}{\partial P_3} \leq 0$ to conclude the proof. Treating $K_\star$ as a function of $P_3$ and $\gamma_\star$, we obtain
	\begin{gather}
		\frac{d K_\star}{d P_3}
		= \frac{\partial K_\star}{\partial P_3} 
		+ \frac{\partial K_\star}{\partial \gamma_\star} 
		\frac{\partial \gamma_\star}{\partial P_3}.
	\end{gather}
	It is straightforward to verify that the first summand is nonnegative. We have already argued that $\gamma_\star$ decreases with $P_3$, and $\frac{\partial K_\star}{\partial \gamma} < 0$ for $\gamma < \gamma_\star$, implying that the second summand is nonnegative as well, completing the proof.
$\square$
\end{proof}


It is easy to compute the optimized iteration count $K_\star(\alpha, \v{\beta})$ and the optimized \rev{constant step-size $\gamma_\star(\alpha, \v{\beta})/\sqrt{K_\star(\alpha, \v{\beta})}$} from Proposition \ref{prop:gamma.K.opt}. The formula is omitted for brevity. Instead, we derive additional insight by fixing $\v{\beta}$ and driving $\alpha$ towards unity. For such an $\alpha, \v{\beta}$, we have
\begin{align} P_2(\alpha, \v{\beta}) \sim (1-\alpha)^{-2}, \ \gamma_\star(\alpha, \v{\beta}) \sim (1-\alpha), \ K_\star(\alpha, \v{\beta}) \sim \frac{1}{\ve^2(1-\alpha)^2}. \end{align}
With $\alpha$ approaching unity, notice that $\PCVaR$ approaches a robust optimization problem. Thus, Algorithm \ref{alg:PDSS} for $\PE'$ is aiming to solve a robust optimization problem via sampling. Not surprisingly, the sample complexity exhibits unbounded growth with such robustness requirements, since we do not assume $\Omega$ to be finite. Also, this growth matches that of solving the SAA problem within $\ve$-tolerance on the unconstrained problem to minimize $\hat{F}(\v{x}) := \frac{1}{K} \sum_{j=1}^K \psi_{\omega^j}^f(\v{x}, u; \alpha)$. To see this, apply Theorem \ref{thm:exp} on $\hat{F}(\v{x})$ with optimized step size from Proposition \ref{prop:gamma.K.opt}, where $P_2 \sim \|\nabla \hat{F}(\v{x})\|^2  \sim (1-\alpha)^{-2}$ and $P_3=1$.


Parallelization can lead to stronger bounds. More precisely, run stochastic approximation in parallel on $N$ machines, each with $K$ samples and compute $ \langle \bar{\v{x}}\rangle_{K+1} := \frac{1}{N} \sum_{j=1}^N \bar{\v{x}}_{K+1}[j]$ using $\bar{\v{x}}_{K+1} [1], \ldots, \bar{\v{x}}_{K+1} [N]$ obtained from the $N$ separate runs. Then, we have  
\begin{align}
\begin{aligned}
&\prob \left\{ G^i \left(\langle \bar{\v{x}}\rangle_{K+1} \right) \geq (1+ \tau) {\eta(\alpha, \v{\beta})}/{\sqrt{K}} \right\} \\
&  \leq \prob \left\{\frac{1}{N} \sum_{j=1}^N \CVaR_{\beta^i}\left[  { g^i_\omega \left(  \bar{\v{x}}_{K+1}[j]  \right)} \right] \geq (1+ \tau) \frac{\eta(\alpha, \v{\beta})}{\sqrt{K}} \right\} \\
& \leq \exp \left(  - \frac{N \tau^2 \eta^2(\alpha, \v{\beta})}{{K [D^i_G]^2}} \right)
\end{aligned}
\label{eq:violation}
\end{align} 
for $i = 1,\ldots,m$ and $\tau > 0$. The steps combine coherence of $\CVaR$, convexity and uniform boundedness of $g^i_\omega$, Hoeffding's inequality and Theorem \ref{thm:cvar}. A similar bound can be derived for suboptimality. Thus, parallelized stochastic approximation produces a result whose $\Ocal(1/\sqrt{K})$-violation occurs with a probability that decays exponentially with  the degree of parallelization $N$.



The bound in \eqref{eq:violation} reveals an interesting connection with results for chance constrained programs. To describe the link, notice that $\CVaR_{\delta}[y_\omega] \leq 0$ implies $\prob \{ y_\omega \leq 0\} \geq 1-\delta$ for any random variable $y_\omega$ and $\delta \in [0, 1)$. Therefore, \eqref{eq:violation} implies
\begin{gather}
\prob \left\{ \prob \left\{ g^i_\omega \left(  \bar{\v{x}}_{K+1} \right) \leq C/{\sqrt{K}} \right\} \geq 1 - \beta^i \text{ is violated}\right\} \leq \exp \left(  - C' / K \right) \leq \nu,
\end{gather}
for constants $C, C'$. Said differently, our stochastic approximation algorithm requires $\Ocal(\log(1/{\nu}))$ samples to produce a solution that satisfies an $\Ocal( 1/\sqrt{\log(1/\sqrt{\nu})})$-approximate chance-constraint with a violation probability bounded by $\nu$. This result bears a striking similarity to that derived in \cite{campi2008exact}, where the authors deterministically enforce $\Ocal(\log(1/{\nu}))$ sampled constraints to produce a solution that satisfies the exact chance-constraint $\prob \left\{ g^i_\omega \left( \v{x} \right) \leq 0 \right\} \geq 1-\beta^i $ with a violation probability bounded by $\nu$. This resemblance in order-wise sample complexity is intriguing, given the significant differences between the algorithms.


\subsection{An illustrative example}
\label{sec:example}

\rev{We explore the use of our algorithm on the following example problem
\begin{equation}
\begin{aligned}
	 \underset{-\frac{1}{2} \le x \le \frac{1}{2}}{\text{minimize}} \ \  \CVaR_{\alpha}\left[ \frac{1}{2} \left(x  - \omega -\frac{1}{2} \right)^2 \right],  \ \text{subject to} \  \CVaR_{\beta}\left[ x + \omega \right] \le 0.
\end{aligned}
\label{eq:example}
\end{equation}
Let $\omega \sim \frac{1}{3}{\sf beta}(2, 2)$ and consider the specific choice of risk parameters $\alpha = 0.3, \beta = 0.2$. To gain intuition into the optimal solution for this example, we numerically estimate $F(x)$ and $G^1(x)$ for each $x$ and plot them in Figure \ref{fig:example.FG}. To that end, we first obtain a million samples of $\omega$. Then, for each value of the decision variable $x$, we sort the objective function value $f_\omega(x)$ and the constraint function value $g_\omega(x)$ with these samples. We then estimate $F$ and $G^1$ as the average of the highest $1-\alpha=70\%$ and  $1-\beta=80\%$ among $f_\omega(x)$'s and $g_\omega(x)$'s, respectively, at each $x$ with those samples. The unique optimum for \eqref{eq:example} is numerically evaluated as $x_\star \approx -0.1929$ for which $F(x_\star) \approx 0.4042$ and $G^1(x_\star) \approx 0$.}

\rev{For this example, it is easy to show that $C_F = \frac{4}{3}$, $C_G = 1$ and $D_G = \frac{5}{6}$ that yields $P_2(0.3, 0.2) = \frac{8276}{93}$ and $P_3(0.3, 0.2) = 50$.
To run Algorithm \ref{alg:PDSS} on $\PE'$ derived from \eqref{eq:example}, we can use constant step-size $\gamma_k = \gamma/\sqrt{K}$ with a pre-determined number of steps $K$ for any $0 < \gamma < P_3^{-1/2}(0.3, 0.2) = \frac{1}{5\sqrt{2}}$. With any given $K$, Theorem \ref{thm:cvar} guarantees that the expected distance to $F(x_\star)$ and the expected constraint violation evaluated at $\bar{{x}}_{K+1}$ decays as $1/\sqrt{K}$. For a given $K$ and $\gamma < \frac{1}{5\sqrt{2}}$, calculating the precise bound $\eta(0.3, 0.2)/\sqrt{K}$ requires the knowledge of $P_1$ or its overestimate. For this example, $| x_\star | \leq \frac{1}{2}$ and $|u^1_\star| \leq D_G = \frac{5}{6}$. Also, $ | u^0_\star |$ is bounded above by the maximum value that $|f_\omega(x)|$ can take, that is given by $\frac{8}{9}$. Since we cannot determine $z_\star$ a priori, we assume $| z_\star | \le 2$ (that will later be shown to be consistent with our result). Starting from $(x_0, u^0_0, u^1_0, z_0) = 0$, we then obtain $P_1=\frac{3197}{81}$.
To solve $\PCVaR$ (or equivalently $\PE'$) with a tolerance of $\ve=5 \times 10^{-3}$, we require $\eta(0.3, 0.2)/\sqrt{K} \leq 5 \times 10^{-3}$. With this tolerance and the values of $P_1, P_2, P_3$, Proposition \ref{prop:gamma.K.opt} yields an optimized
$\gamma_\star = 0.0808$ and $K_\star \approx 1.35 \times 10^{9}$. We run Algorithm \ref{alg:PDSS} on $\PE'$ with constant step-size $\gamma_\star/\sqrt{K_\star}$ and plot $F$ and $G^1$ at the running ergodic mean of the iterates, i.e., at $\bar{x}_k := \frac{1}{k}\sum_{j=1}^k x_j$ for each $k$. Again $F$ and $G^1$ are evaluated numerically using the $\CVaR$-estimation procedure we outlined above.}

\begin{figure}[t]
	\centering
	\subfloat[]{
		\includegraphics[width=0.5\textwidth, trim=15 0 15 0, clip]{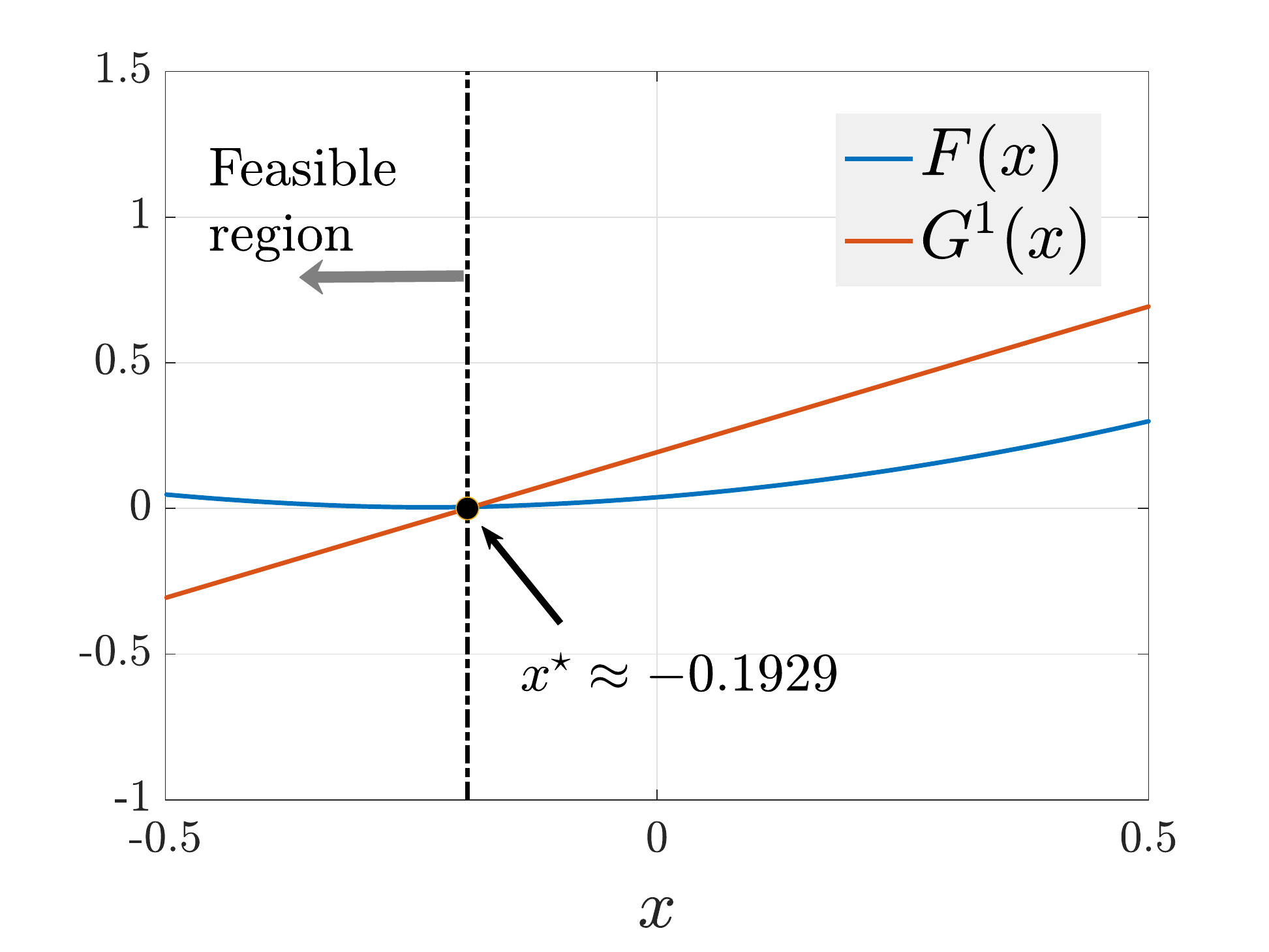}
		\label{fig:example.FG}
	} 
	\subfloat[]{
		\includegraphics[width=0.5\textwidth]{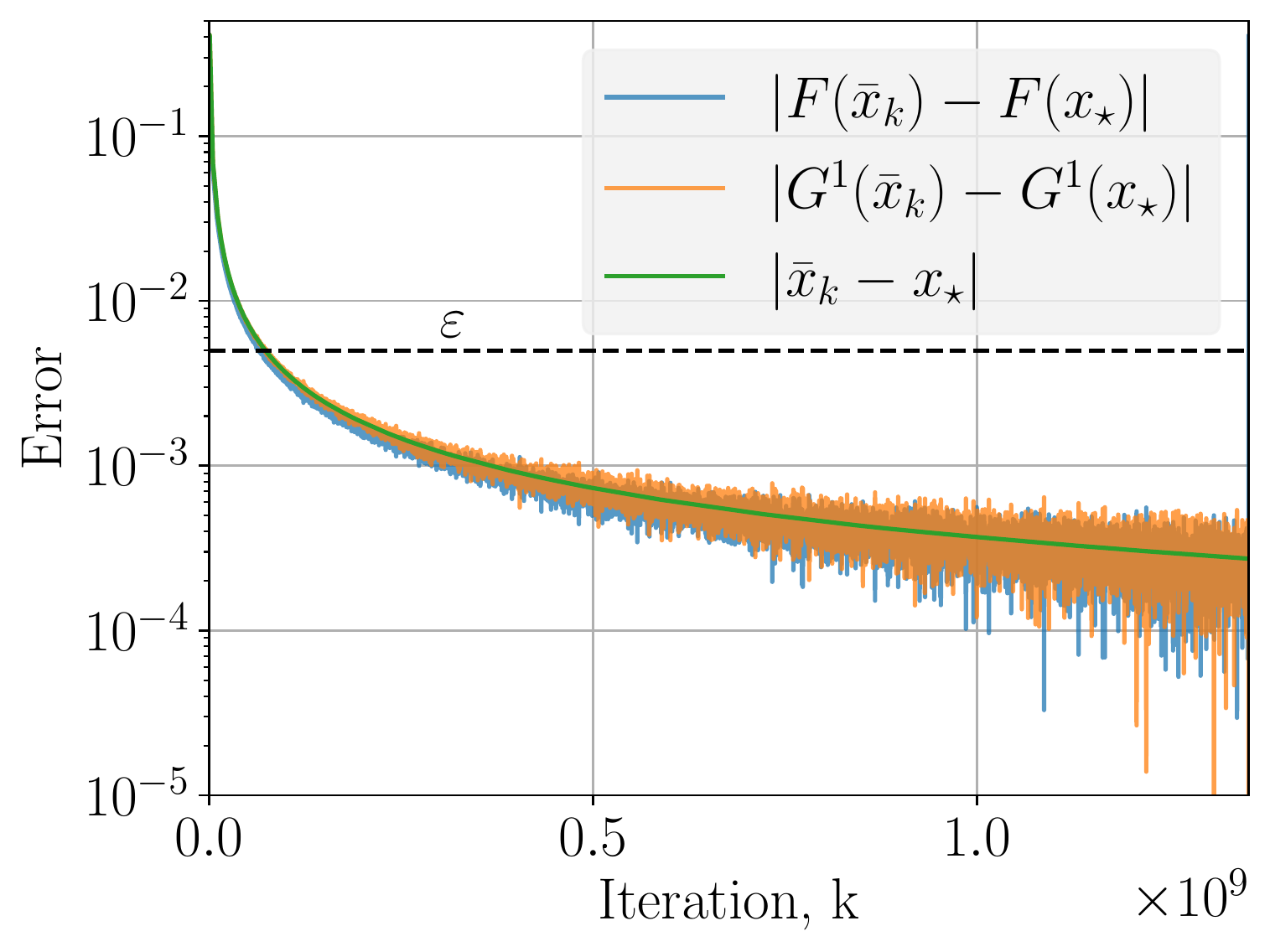}
		\label{fig:example.convergence}
	} 
	\caption{Plots of 
	\protect\subref{fig:example.FG} numerically estimated $F$ and $G^1$ over $\Xset = [-\frac{1}{2}, \frac{1}{2}]$, and \protect\subref{fig:example.convergence}  convergence of the running ergodic mean and $F, G$ evaluated at the mean for the example problem \eqref{eq:example} with $\alpha = 0.3$, $\beta = 0.2$. 
	}
\end{figure}

\rev{Notice that Theorem \ref{thm:cvar} only guarantees a bound on $F (\bar{x}_{K_\star+1}) - F(\bar{x}_\star)$ and $G^1(\bar{x}_{K_\star+1})$ in expectation. Thus, one would expect that only the average of the $\CVaR$ of $F$ and $G^1$ evaluated at $\bar{x}_{K_\star + 1}$ over multiple sample paths to respect the $\ve$-bound. However, our simulation yielded $\bar{x}_{K_\star + 1} = -0.1926$ and $\bar{z}_{K_\star + 1} = 0.8976$, for which
\begin{gather}
\begin{gathered}
    F(\bar{x}_{K_\star + 1}) \approx 0.4040 \leq F(x_\star) + \ve \approx 0.4042 + 0.0050 = 0.4092, \\
    G^1(\bar{x}_{K_\star + 1}) \approx 0.0002 \leq G^1(x_\star) + \ve \approx 0 + 0.0050 = 0.0050,
\end{gathered}
\end{gather}
i.e., the ergodic mean after $K_\star$ iterations respects the $\ve$-bound over the plotted sample path. The same behavior was observed over multiple sample paths. The ergodic mean of the dual iterate is indeed consistent with our assumption $|z_\star| \leq 2$ made in deriving $\eta(0.3, 0.2)$. We point out that the ergodic mean in Figure \ref{fig:example.convergence} moves much more smoothly than our evaluation of $F$ and $G^1$ at those means, especially for large $k$. The noise in $F$ in $G^1$ emanate from the finitely many samples we use to evaluate $F$ and $G^1$. The errors appear much more pronounced at larger $k$, given the logarithmic scale of the plot.
}

\rev{The optimized iteration count $K_\star(\alpha, \beta)$ from Proposition \ref{prop:gamma.K.opt} with a modest $\alpha=0.3, \beta=0.2$ is quite high even for this simple example. This iteration count only grows with increased risk aversion as Figure \ref{fig:example.gamma.K.ab} reveals. 
Figure \ref{fig:example.convergence} suggests that the $\ve=5\times 10^{-3}$ tolerance is met  earlier than $K_\star$ iterations. This is the downside of optimizing upper bounds to decide step-sizes for subgradient methods. Carefully designed termination criteria may prove useful in practical implementations.}

\begin{figure}[ht]
	\centering
	\includegraphics[width=0.95\textwidth,trim=75 0 75 0,clip]{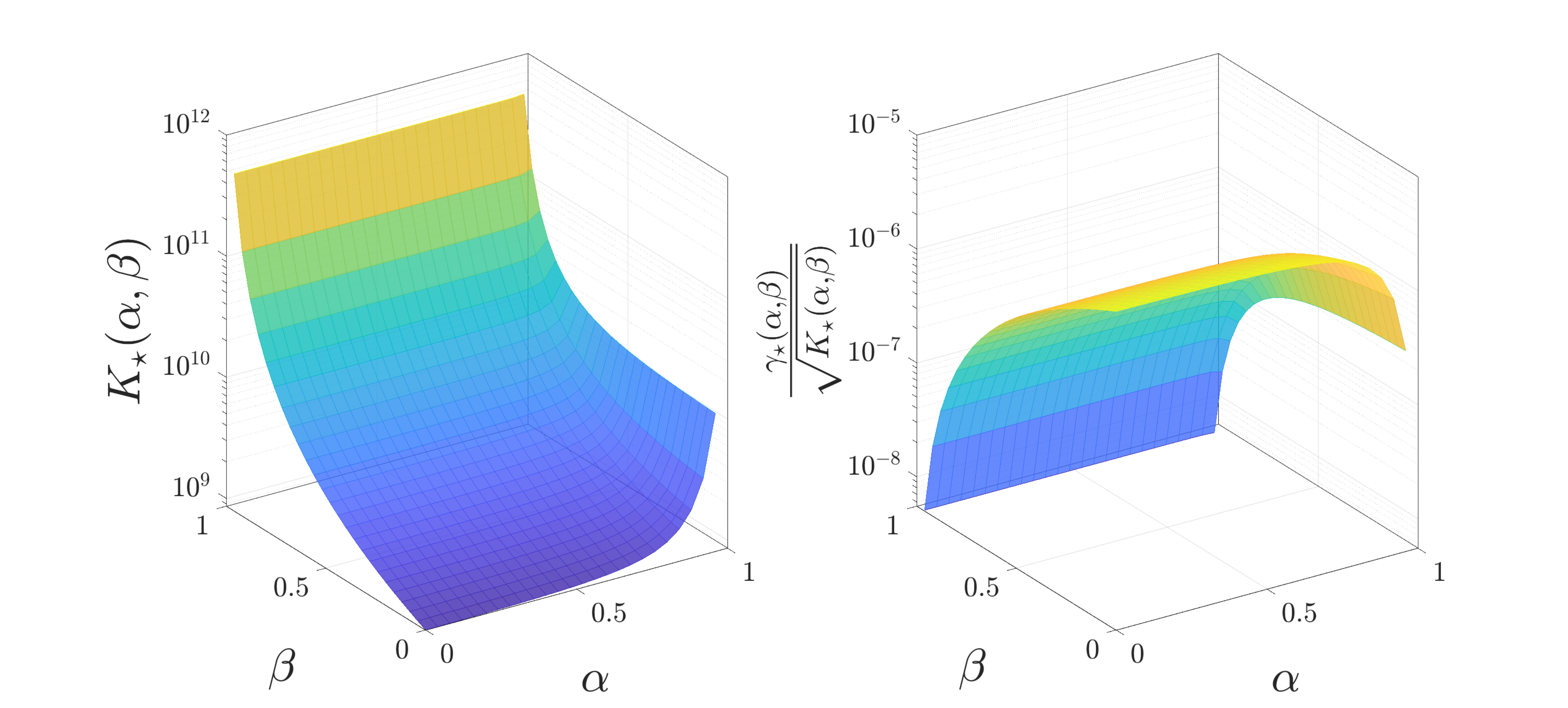}
	\caption{Plot of the optimized number of iterations $K_\star(\alpha, {\beta})$ on the left and the optimized step size $\gamma_\star(\alpha, \beta)/\sqrt{K_\star(\alpha, \beta)}$ on the right to achieve a tolerance of $\ve=5 \times 10^{-3}$ for the example problem in \eqref{eq:example}.}
	\label{fig:example.gamma.K.ab}
\end{figure}

\rev{We end the numerical example with a remark about the comparison of Algorithm \ref{alg:PDSS} that uses Gauss-Seidel-type dual update in \eqref{eq:zUpdate} and another that uses the popular Jacobi-type dual update on $\PE'$ for \eqref{eq:example} with $\alpha=0.3, \beta=0.2$. This alternate dual update replaces $\v{g}_{\omega_{k +1/2}}(\v{x}_{k+1})$ in \eqref{eq:zUpdate} by $\v{g}_{\omega_{k}}(\v{x}_{k})$. That is, the same sample $\omega_k$ is used for both the primal and the dual update. And, the primal iterate $\v{x}_k$ is used instead of $\v{x}_{k+1}$ to update the dual variable. We numerically compared this primal-dual algorithm with Algorithm \ref{alg:PDSS} with various choices of step-sizes (consistent with the requirements of Theorem \ref{thm:cvar}) and iteration count for our example and its variations. For each run, we found that the iterates from both these algorithms moved very similarly. The differences are too small to report. The Jacobi-type update requires half the number of samples compared to Algorithm \ref{alg:PDSS}. While the extra sample helps us in the theoretical analysis, our experience with this stylized example does not suggest any empirical advantage. A more thorough comparison between these algorithms, both theoretically and empirically, is left to future work.}

\section{Conclusions and future work}
\label{sec:conc}

In this paper, we study a stochastic approximation algorithm for $\CVaR$-sensitive optimization problems. Such problems are remarkably rich in their modeling power and encompass a plethora of stochastic programming problems with broad applications. 
We study a primal-dual algorithm to solve that problem that processes samples in an online fashion, i.e., obtains samples and updates decision variables in each iteration. Such algorithms are useful when sampling is easy and intermediate approximate solutions, albeit inexact, are useful. The convergence analysis allows us to optimize the number of iterations required to reach a solution within a prescribed tolerance on expected suboptimality and constraint violation. The sample and iteration complexity predictably grows with risk-aversion. Our work affirms that a modeler must not only consider the attitude towards risk but also consider the computational burdens of risk in deciding the problem formulation.

Two possible extensions are of immediate interest. First, primal-dual algorithms find applications in multi-agent distributed optimization problems over a possibly time-varying communication network. We plan to extend our results to solve distributed risk-sensitive convex optimization problems over networks, borrowing techniques from \cite{nedic2009distributed,dominguez2013distributed}.
Second, the relationship to sample complexity for chance-constrained programs in \cite{campi2008exact} encourages us to pursue a possible exploration of stochastic approximation for such optimization problems.



\section{Acknowledgements}
We thank Eilyan Bitar, Rayadurgam Srikant, Tamer Ba\c{s}ar and Stan Uryasev for helpful discussions. This work was partially supported by the International Institute of Carbon-Neutral Energy Research (I$^2$CNER) and the Power System Engineering Research Center (PSERC).

\bibliographystyle{unsrt}
\bibliography{cvar_oco}

\end{document}